\documentclass[letterpaper,12pt]{amsart}

\usepackage[all]{xy}                        %

\usepackage[margin=1in]{geometry}

\CompileMatrices                            % Faster

\UseTips                                    % Use

\input xypic
\usepackage[bookmarks=true]{hyperref}       % Hyperref
\usepackage{amssymb,latexsym,amsmath,amscd}
\usepackage{xspace}
\usepackage{color}
\usepackage{xcolor}
\usepackage{graphicx}
\usepackage{caption}

\usepackage{mathtools}
\usepackage{tikz}
\usetikzlibrary{cd}

\usepackage[utf8]{inputenc}
\usepackage{fourier}
\usepackage{array}
\usepackage{makecell}

\usepackage[foot]{amsaddr}

\usepackage{array}
\newcolumntype{L}[1]{>{\raggedright\let\newline\\\arraybackslash\hspace{0pt}}m{#1}}
\newcolumntype{C}[1]{>{\centering\let\newline\\\arraybackslash\hspace{0pt}}m{#1}}
\newcolumntype{R}[1]{>{\raggedleft\let\newline\\\arraybackslash\hspace{0pt}}m{#1}}

\reversemarginpar

\theoremstyle{plain}
\newtheorem{theorem}{Theorem}[section]
\newtheorem*{theorem*}{Theorem}
\newtheorem{proposition}[theorem]{Proposition}
\newtheorem{corollary}[theorem]{Corollary}
\newtheorem{lemma}[theorem]{Lemma}

\newtheorem*{conjecture*}{Conjecture}

\theoremstyle{definition}
\newtheorem{definition}[theorem]{Definition}

\newtheorem{question}[theorem]{Question}

\newtheorem{remark}[theorem]{Remark}
\newtheorem{example}[theorem]{Example}

\newcommand{\enm}[1]{\ensuremath{#1}}          %
\newcommand{\op}[1]{\operatorname{#1}}
\newcommand{\cal}[1]{\mathcal{#1}}

\renewcommand{\bar}[1]{\overline{#1}}

\newcommand{\CC}{\enm{\mathbb{C}}}

\newcommand{\ZZ}{\enm{\mathbb{Z}}}
\newcommand{\FF}{\enm{\mathbb{F}}}

\newcommand{\PP}{\enm{\mathbb{P}}}

\newcommand{\Ee}{\enm{\cal{E}}}
\newcommand{\Ff}{\enm{\cal{F}}}
\newcommand{\Gg}{\enm{\cal{G}}}

\newcommand{\Ii}{\enm{\cal{I}}}

\newcommand{\Kk}{\enm{\cal{K}}}
\newcommand{\Ll}{\enm{\cal{L}}}

\newcommand{\Oo}{\enm{\cal{O}}}

\newcommand{\Qq}{\enm{\cal{Q}}}

\newcommand{\Tt}{\enm{\cal{T}}}

\renewcommand{\phi}{\varphi}
\renewcommand{\theta}{\vartheta}
\renewcommand{\epsilon}{\varepsilon}

\newenvironment{nscenter}
 {\parskip=0pt\par\nopagebreak\centering}
 {\par\noindent\ignorespacesafterend}

\newcommand{\Hom}{\op{Hom}}

\newcommand{\Ext}{\op{Ext}}

\renewcommand{\to}[1][]{\xrightarrow{\ #1\ }}

\newcommand{\old}[1]{}

\DeclareMathOperator{\hh}{h}

\DeclareMathOperator{\di}{dim}

\title[Generalized logarithmic sheaf on smooth projective surfaces]{Generalized logarithmic sheaf on smooth projective surfaces}

\author{S.~Huh $^{1}$} 
\address{$^{1}$Sungkyunkwan University, Suwon 440-746, Korea}
\email{sukmoonh@skku.edu}

\author{S.~Marchesi $^{2}$}
\address{$^{2}$ Universitat de Barcelona, 08007 Barcelona, Spain}
\address{$^{2}$ Centre de Recerca Matemàtica,  08193 Barcelona, Spain}
\email{marchesi@ub.edu}

\author{J.~Pons-Llopis $^{3}$}
\address{$^{3}$ Politecnico di Torino, 10129 Torino, Italy}
\email{juan.ponsllopis@polito.it}

\author{J.~Vall\`es $^{4}$}
\address{$^{4}$ Universit\'e de Pau et des Pays de l'Adour, 64013 Pau Cedex, France}
\email{jean.valles@univ-pau.fr}

\thanks{SH is supported by the National Research Foundation of Korea(NRF) grant funded by the Korea government(MSIT) (No. 2018R1C1A6004285 and No. 2016R1A5A1008055). SM is partially supported by PID2020-113674GB-I00 and by the Spanish State Research Agency, through the Severo Ochoa and María de Maeztu Program for Centers and Units of Excellence in R\&D (CEX2020-001084-M) and is a member of INdAM– GNSAGA. SM and JPLl are members of INdAM– GNSAGA. JV is partially supported by PRCI ANR -- FAPESP : ANR-21-CE40-0017}


\begin{document}

\begin{abstract}
We define the notion of generalized logarithmic sheaves on a smooth projective surface, associated to a pair consisting of a reduced curve and some fixed points on it. We then set up the study of the Torelli property in this setting, focusing mostly in the case of the blow-up of the projective plane on a reduced set of points and, in particular, in the case of the cubic surface.  We also study the stability property of generalized logarithmic sheaves as well as carrying out the description of their moduli spaces.

\end{abstract}

\maketitle
\tableofcontents
\bigskip

\section{Introduction}

Given a smooth projective variety $X$ and a reduced and effective divisor $D$ with normal crossings on it, one of the sheaves that have raised particular attention during the recent years is the {\it logarithmic sheaf} or the {\it sheaf of logarithmic differentials} $\Omega_X^1(\log D)$, defined as the  sheaf of differential $1$-forms on $X$ with logarithmic poles along $D$. Its dual has also a very classical presentation:  it is the {\it logarithmic tangent sheaf} associated to $D$, and denoted by $\Tt_X(-\log D)$, which can be defined as the sheaf whose local sections are  vector fields tangent to the divisor $D$. The study of logarithmic sheaves was initiated by Deligne in \cite{De} having in mind to define a mixed Hodge structure on the complement $X\backslash
 D$. Later on, it was generalized by Saito to arbitrary reduced divisors in \cite{Saito}. 

Related to these vector bundles a very natural question arises immediately, namely, to which degree the logarithmic sheaf $\Omega_X^1(\log D)$ determines the divisor $D$. Obviously, two extremal cases can appear: on one extreme, it would be possible that the logarithmic sheaf determines unambiguously the original divisor. In this case, by analogy with the famous result about the injectivity of the period map from the moduli space of algebraic curves of fixed genus, we will say that \emph{the Torelli property holds} for the divisor $D$.  Among other upshots, such a property often allows to give a nice description of the sheaves in the corresponding moduli space. The opposite extreme would appear when $\Omega_X^1(\log D)$ splits as a direct sum of line bundles of the form $\bigoplus_{i=1}^{\di (X)}\Oo_X(a_i)$. In this case $D$ is called a \textit{free divisor} or analogously, we will say that its associated logarithmic sheaf is free.

Both situations have been identified in different settings: for instance, in \cite{UY}, Ueda and Yoshinaga showed that the Torelli property holds for a smooth divisor in $\PP^n$ if and only if it is not of Sebastiani-Thom type (in the sense that its defining polynomial has a presentation as the sum of two polynomials on disjoint sets of variables). For non-smooth divisors, one may call a divisor an {\it arrangement} of its irreducible components, depending on view points. Dolgachev and Kapranov showed in \cite{DK} that a general arrangement of $m$ hyperplanes in $\PP^n$ is determined by the logarithmic sheaf whenever $m \ge 2n+3$, unless the hyperplanes osculate the same rational normal curve. The result was extended to the range $m \ge n+2$ in \cite{Valles}. For higher degree, Angelini showed in \cite{Angelini} that the Torelli property holds for a generic arrangement of $n$ conics in $\PP^2$ whenever $n\ge 9$. For irreducible singular divisors, Faenzi and the second named author proved in \cite{FM} that divisors coming from generic determinants have
simple (in characteristic zero, stable) logarithmic tangent sheaves and satisfy the Torelli property.\\ Regarding the freeness of logarithmic tangent vector bundles, a leading conjecture by Terao, reported in \cite{OT}, states that given two hyperplane arrangements $A$ and $A'$ in $\PP^n$ with the same combinatorial type, if $\Omega^1_{\PP^n}(\log A)$ is free then the same should hold for $\Omega^1_{\PP^n}(\log A')$ with the same splitting type.\\

The goals of the present paper are mainly twofold: firstly, given a finite set of points $Z\subset X$ on a surface $X$ and a curve $D \subset X$ containing $Z$, we introduce a natural generalization of $\Omega_X^1(\log D)$ as a particular extension, denoted by $\Omega_X^1(\log (D, Z))$, of the ideal sheaf $\Ii_{Z, D}$, of $Z$ inside $D$, by the cotangent bundle $\Omega_X^1$, such that it is singular exactly along $Z$. We will call $\Omega_X^1(\log (D, Z))$ the \emph{generalized logarithmic sheaf}   associated to the pair $(D,Z)$. Notice that when $Z=\emptyset$, one gets the usual logarithmic vector bundle $\Omega_X^1(\log (D, \emptyset))=\Omega_X^1(\log D)$. Then we propose a study of the \emph{Torelli property} in this setting: namely, in which cases, for a divisor $D \subset X$ containing $Z$, the isomorphism class of $\Omega_X^1(\log (D, Z))$,
as an $\Oo_X$-module, determines $(D,Z)$ uniquely.

This construction leads to the second goal of the present paper. Indeed, if we consider the blow-up $\pi:\widetilde{X}\rightarrow X$ along an arbitrary finite set $Z'$ and the strict transform $\widetilde{D}\subset \widetilde{X}$, it turns out that the push-forward $\pi_*\Omega_{\widetilde{X}}^1(\log \widetilde{D})$ is exactly the aforementioned sheaf $\Omega_X^1(\log (D, Z))$, where $Z=D\cap Z'$. Therefore, as soon as the Torelli property holds for $D \subset X$, then the Torelli property holds also for $\widetilde{D}$ in $\widetilde{X}$, but the reciprocal statement can be false. This prompted us to take up the study of the (in)dependency of the Torelli property when we change a divisor $D\subset X$ by its strict transform $\widetilde{D}$ in a given blow-up $\widetilde{X}$ along $Z'$ and in particular to underline the similarities and differences between the case when $Z'\subset D$ (in Section \ref{sec-projplane}) and the case when $Z'\cap D=\emptyset$ (in Section \ref{sec-blowup}).

In order to launch the research following this strategy, we decided in this paper to focus on the first natural case that one can consider, namely on
(generalized) logarithmic sheaves on blow-ups of $\PP^2$ and, in particular, on cubic surfaces $S\subset\PP^2$. In the first aforementioned case, namely when we consider just set of points in $\PP^2$ contained in the divisors under consideration, we managed to extend the main result from \cite{UY2} and prove that the Torelli property does not hold for pairs $(D,Z)$ consisting of a smooth plane cubic curve $D$ of Sebastiani-Thom type and a finite set of points $Z\subset D$ unless in the trivial situation when $Z$ already determines the curve itself (see Theorem \ref{STP}).  On the other hand, when the set of points has empty intersection with the curve, we placed ourselves into the classical setting of cubic curves $S\subset\PP^3$ seen as the blow-up $S=\mathrm{Bl}_Z{\PP^2}\rightarrow\PP^2$ of six general points $Z$ in $\PP^2$ and proved that the linear systems of strict transforms of lines $|\Oo_S(L)|$ (see Theorem \ref{torr-k}), conics $|\Oo_S(2L)|$ (see Proposition \ref{quad}) and cubic curves $|\Oo_S(3L)|$ (see Proposition \ref{prop52}) satisfy the Torelli property. In this setting, we proved that the logarithmic bundle $\Omega_S^1(\log D)$ with $D\in|\Oo_S(L)|$ (resp. $D\in|\Oo_S(2L)|$) is a $\mu$-stable vector bundle with respect to the hyperplane section $\Oo_{S}(1)$ with Chern classes $(K_S+L,7)$ (resp. $(K_S+2L,7)$) and therefore the induced map from the complete linear system $\PP \mathrm{H}^0(\Oo_S(D))$ to the respective moduli space of $\mu$-stable bundles with fixed Chern classes is generically one-to-one (cf. Theorem \ref{stab1}, resp. Theorem \ref{stab2})).
\\

Let us explain now the structure of the paper.

In Section \ref{preliminaries}, after recalling the required concepts from the theory of logarithmic sheaves and fixing the notation, we prove a key technical result (see Lemma \ref{keylem}) which describes the restriction of logarithmic tangent sheaves to a given curve. This result will reveal itself of extreme importance in the subsequent parts of the work in order to determine the splitting type of a given sheaf when restricted to a rational curve. Furthermore, in the case where $X$ is a smooth projective surface we introduce the notion of \textit{generalized logarithmic sheaf} $\Omega^1_X(\log(D,Z))$ on $X$ associated to the pair $(D,Z)$, being $Z$ a fixed set of points in the divisor $D$. We conclude the section studying the basic properties of this sheaf.

In Section \ref{toreliness} we are going to set up the main problem we are interested in, namely, up to which point the generalized logarithmic sheaf $\Omega^1_X(\log(D,Z))$ determines the pair $(D,Z)$. By analogy with other well-studied settings, we are going to call it the {\it Torelli property} for generalized logarithmic sheaves. Moreover, we are going to see how the Torelli property is affected under blow-ups, depending on whether the blow-up is done on points contained (in Subsection \ref{subsec-pointsin}) or not (in Subsection \ref{subsec-pointsout}) in the divisors under consideration.

The first situation, namely generalized logarithmic sheaves $\Omega^1_X(\log(D,Z))$ 
with $Z\subset D$ is studied deeply in Section \ref{sec-projplane} for the particular case of $X$ being the projective plane. The goal of this part is to give a complete picture for generalized logarithmic sheaves associated to a plane curve $D$ with a fixed set of points.
Recall that it is already known that the Torelli property holds for a curve $D$ of degree at least $3$ and not of Sebastiani-Thom type; see \cite{UY}. In particular we prove that:
\begin{itemize}
    \item the Torelli property holds for a line with at least two of its points fixed (so determining the line itself), and
    \item the Torelli property holds for a conic with at least three of its points fixed; see Theorem \ref{torr-k}, and
    \item the Torelli property holds for a cubic curve of Sebastiani-Thom type  if and only if the fixed points determine the curve; see Theorem \ref{STP}. 
\end{itemize}

Finally, in Section \ref{sec-blowup} we tackle the opposite situation, i.e., we deal with (generalized) logarithmic sheaves of the strict transform of curves $C\subset \PP^2$  under blow-ups on set of points $Z\subset\PP^2$ with $Z\cap C=\emptyset$. In particular, if $S\subset\PP^3$ is a smooth cubic surface and $D$ belongs to the linear systems $\mid \Oo_S(L) \mid$ or $\mid \Oo_S(2L) \mid$, being $L$ the pullback of a line in $\PP^2$ throughout the blow-up map, we prove (Theorem \ref{stab1} and Theorem \ref{stab2}) that:
\begin{itemize}
    \item the logarithmic bundle $\Omega_S^1(\log D)$ is stable, with respect to the polarization $H$  given by the hyperplane section, with Chern classes $(c_1, c_2)=\left(K_S+jL, 7\right)$. Furthermore, the Torelli property holds for $D$. Hence, the induced rational map between the moduli spaces of divisors and of stable sheaves
\[
\Psi_j : \PP \mathrm{H}^0(\Oo_S(jL)) \dashrightarrow \mathbf{M}_S^H(K_S+jL, 7),
\]
for $j=1,2$, is generically one-to-one. 
\end{itemize}

\noindent\textbf{Acknowledgements.} All the authors would like to thank Sungkyunkwan University for its hospitality and for providing the best working conditions. They also want to thank the referees for many useful comments.

%%%%%%%%%%%%%%%%%%%%%%%%%

\section{Preliminaries}\label{preliminaries}

Throughout the paper we are going to work over the field of complex numbers. Let us start recalling the construction of logarithmic sheaves; see \cite{Saito} for more details.
Let $X$ be a smooth irreducible projective variety of dimension $n\ge 2$, and let $D=D_1+\dots + D_m$ be an effective and reduced divisor on $X$ with simple normal crossings, where each $D_j$ is irreducible, i.e., each $D_j$ is smooth and they intersect transversally: we also denote $\cup_{j=1}^m D_j$ by $D$. Then one can consider the sheaf of differential $1$-forms with logarithmic poles along $D$, called the {\it logarithmic sheaf} on $X$ associated to $D$, denoted by $\Omega_X^1(\log D)$. This sheaf turns out to be locally free, and the sections of $\Omega_X^1(\log D)$ around a point $x\in X$ are meromorphic $1$-forms of the form $\omega+\sum_{i=1}^k u_i \left(\frac{d z_i}{z_i}\right)$, where $z_1\cdots z_k=0$ is the local defining equation of $D$. Here, $\omega$ is holomorphic $1$-form and $u_i$'s are holomorphic functions near $x$. The Poincar\'e residue morphism is locally defined by
\[
\begin{array}{llll}
\mathrm{res}: & \hspace{2cm}\Omega_X^1(\log D) & \rightarrow & \oplus_{j=1}^m\Oo_{D_j}\\ 
       &  \sum_{i=1}^k u_i \left(\frac{d z_i}{z_i}\right) + \sum_{i=k+1}^n u_i dz_i & \mapsto  &(u_1(x),\dots, u_k(x), 0, \dots 0), 
\end{array}
\]
and it fits into the canonical exact sequence
\begin{equation}\label{sseeq1}
0\to \Omega_X^1 \to \Omega_X^1(\log D) \stackrel{\mathrm{res}}{\xrightarrow{\hspace*{.5cm}} } \oplus_{j=1}^m \Oo_{D_j} \to 0.
\end{equation}
The dual of $\Omega_X^1(\log D)$ is called the {\it logarithmic tangent sheaf} associated to $D$, denoted by $\Tt_X(-\log D)$; it is the sheaf of holomorphic vector fields tangent to $D$. Then, dualising the sequence (\ref{sseeq1}), we get the following:
\begin{equation}\label{sseeq3}
0\to \Tt_X(-\log D) \to \Tt_X \stackrel{\psi}{\longrightarrow} \oplus_{j=1}^m \Oo_{D_j}(D_j) \to 0.
\end{equation}

\begin{remark}
Before stating Lemma \ref{keylem} let us spell out here the sequence (\ref{sseeq3}). Take an open subset $U \subset X$ and recall that, denoting by $\phi_U$ the equation defining $D$ locally in $U$, we have
\[
\Tt_X(-\log D)(U) = \left\{\partial \in \mathrm{Der} \left(\Oo_X(U),\Oo_X(U)\right) \: |\: \partial(\phi_U) \in ( \phi_U) \right\}. 
\]
Here, $( \phi_U)$ is the ideal in $\Oo_X(U)$ generated by $\phi_U$. With this setting, the map $\psi$ in the sequence (\ref{sseeq3}) is defined by $\psi(\partial) = \left[\partial(\phi_U)\right]$, where $[-]$ denotes the equivalent class in $\Oo_D(U)$.
\end{remark}

\subsection{The Key Restriction Lemma}

We will see that one of our main concerns is to compare logarithmic sheaves. In order to do that, one of our main strategies is to study the restriction of the logarithmic sheaf to embedded curves. The following lemma will be our central tool.

\begin{lemma}[Key Restriction Lemma]\label{keylem}
Let $X$ be a smooth variety and $D$ an effective and reduced divisor on $X$ with simple normal crossings. Denoting by $D_{\mathrm{sm}}$ the smooth locus of the divisor, consider a smooth curve $i: C \hookrightarrow X$ such that
\[
D \cap C = \sum_{i=1}^p a_i P_i, \mbox{ with } P_i \in D_{\mathrm{sm}} \mbox{ and } a_i\in \mathbb{Z}_{>0},
\]
and the subscheme $B = (D \cap C)_{\mathrm{red}}$. Then, there exists a commutative diagram
\begin{equation}\label{keydiag}
\begin{tikzcd}[
 column sep=small,row sep=small,
  ar symbol/.style = {draw=none,"\textstyle#1" description,sloped},
  isomorphic/.style = {ar symbol={\cong}},
  ]
 &0   \ar[d] &0 \ar[d]  &0 \ar[d]\\
0\ar[r] &  \Tt_C(-\log B) \ar[r]\ar[d, "\alpha~", labels=left]& i^*\left(\Tt_X(-\log D)\right) \ar[r] \ar[d] & \mathcal{N}_{\log D/B} \ar[r]\ar[d] & 0\\
0\ar[r] & \Tt_C  \ar[r, "\beta"]\ar[d] &i^*\Tt_X  \ar[d,"\psi_C~",labels=left] \ar[r] &\mathcal{N}_{C|X}  \ar[r]\ar[d]& 0\\
0\ar[r] &\Oo_B\ar[r] \ar[d]&  \oplus_{j=1}^m i^*\Oo_{D_j}(D_j) \ar[r] \ar[d] & \Oo_{\widetilde{B}}\ar[r]\ar[d]& 0\\
&0  &0   &0
\end{tikzcd}
\end{equation}
with $\widetilde{B}=\sum_{i=1}^p (a_i-1)P_i$. We will call the coherent sheaf $\mathcal{N}_{\log D/B}$ the logarithmic normal sheaf of $C\subset X$ with respect to the divisor $D$. Furthermore, if we consider a proper subset $B \subsetneq (D \cap C)_{\mathrm{red}}$ the statement is false, i.e, to get a commutative diagram as (\ref{keydiag}), we must consider all the (reduced) points given by the intersection of the divisor $D$ with the curve $C$.

\end{lemma}

\begin{proof}
Since we are dealing with a local problem, we can restrict ourselves to the case when $p=1$, namely $D \cap C=aP$ and $B=\{P\}\subset C$. Let us consider an affine open neighborhood $U=\mathrm{Spec}\left(  \CC [x_1, \dots, x_n]/I \right)\subset X$ of $P=(0,\dots, 0)$ for which 
\[
C\cap U=V(x_2,\dots, x_n) \text{ and }D\cap U=V(f(x_1,\dots,x_n))
\]
with $f\in \CC[x_1, \dots, x_n]/I$. From the hypothesis that $D \cap C\cap U=aP$ with $a\geq 1$, we see that $f(x_1,0,\dots,0)=x_1^a$, and in particular $f$ has a presentation of the form:
\[
f(x_1,\dots, x_n)= x_1^a+\sum_{j=2}^s x_j \cdot f_j(x_1,\dots,x_n)
\]
Let us make now explicit the map $\psi_C$ in the diagram (\ref{keydiag}) which is induced from $\psi$ in (\ref{sseeq3}). It is defined as
\[
\begin{array}{rccc}
\psi_C(U): = \psi \otimes id : & \Tt_X(U) \otimes \Oo_C(U) & \longrightarrow  & \Oo_{D}(D)(U) \otimes \Oo_C(U)\\
& \partial \otimes a & \mapsto &\left[a\partial(f)\right]
\end{array}
\]
where $f$ denotes the local equation of $D$ in $U$ and, in this case, $[-]$ denotes the equivalent class in $\Oo_{D \cap C}(U)$, i.e., modded out by the ideal $(f,x_2,\dots,x_n)$. We will now prove that $\psi_C\circ\beta\circ\alpha=0$. Because of the natural inclusion $\Tt_C \subset i^*(\Tt_X)$, we can choose the local coordinates in $U\cap C$ to be one of the local coordinates of $U$ inside $X$ so that
\[
\alpha(\Tt_C(-\log B)(U))=\CC \{x_1 \partial_{x_1}\}\subset \Tt_C(U).
\]
Thus we have
\[
x_1 \partial_{x_1}(f)=ax_1{\cdot}x_1^{a-1}+x_1\left(\sum_{j=2}^s x_j\partial_{x_1}f_j\right)
\]
which, after modded out by $(f,x_2,\dots,x_n)=(x_1^a,x_2,\dots, x_n)$, is zero. This means that we have an inclusion
\[
\Tt_C(-\log B) \hookrightarrow i^*\left(\Tt_X(-\log D)\right)
\]
which gives the following commutative diagram
\[
\begin{tikzcd}[
 column sep=small,row sep=small,
  ar symbol/.style = {draw=none,"\textstyle#1" description,sloped},
  isomorphic/.style = {ar symbol={\cong}},
  ]
 &0   \ar[d] &0 \ar[d]  \\
0\ar[r] &  \Tt_C(-\log B) \ar[r]\ar[d]& i^*\left(\Tt_X(-\log D)\right) \ar[d] & \\
0\ar[r] & \Tt_C  \ar[r]\ar[d] &i^*\left(\Tt_X\right) \ar[d,"~\psi_C"] & \\
 & \Oo_B \ar[r,"h"] \ar[d]&  \oplus_{j=1}^m i^*\left(\Oo_{D_j}(D_j)\right)  \ar[d] &\\
&0  &0.
\end{tikzcd}
\]
It remains to prove that $h$ is injective. In order to show this, recall that the open subset $U\subset X$ satisfies 
\[
B\cap U = P \mbox{ and } C\cap D \cap U = aP,
\]
having both schemes  as support a single point. In this case, the restriction map $h_{|U}$ is either injective or zero. The latter case is impossible; then we would have that $\psi_C\circ\beta=0$. However, considering the element $\partial_{x_1}\in \Tt_C(U)$ that can be seen as well as one element in a basis for $\Tt_X(U)$, we have
\[
\partial_{x_1}f\equiv x_1^{a-1}\not\equiv 0 \hspace{.5cm}(~\mathrm{mod} (x_1^a,x_2,\dots ,x_n)).
\]

To prove the final part of the statement, we can restrict once again to the case $D\cap C = aP$, but now being $B$ empty. In this case, Diagram (\ref{keydiag}) would imply an injection $\Tt_C \hookrightarrow i^*\left(\Tt_X(-\log D)\right)$. This is clearly a contradiction because, using the same setting and notation as in the first part of the proof, the derivation $\partial_{x_1}$ belongs to $\Tt_C(U)$ but not to $i^*\left(\Tt_X(-\log D)\right)(U)$.

\end{proof}

\begin{remark}
Assume that $\dim X=2$, and so the logarithmic tangent sheaf $\Tt_X(-\log D)$ is always locally free. Then the proof of Lemma \ref{keylem} shows that the condition of Lemma \ref{keylem} can be weakened. Indeed, $D$ does not necessarily have to have simple normal crossings and $C$ can be chosen to be an arbitrary smooth curve so that the intersection points $P_i$'s do not have to be in $D_{\mathrm{sm}}$. 
\end{remark}

To underline the importance of the previous result, we will recover, by means of Lemma \ref{keylem}, well known descriptions of logarithmic tangent sheaves. In particular, we will see how to recover the vector bundle by studying its restriction to lines and, more in general, to rational curves. Thanks to a well known result due to Grothendieck (see \cite{Groet} and \cite[I.2-I.3]{OSS}), a vector bundle $\mathcal{E}$ on $\PP^1$ splits as a sum of line bundles. Therefore, given a vector bundle $\mathcal{E}$ on $X$ and a line $L \subset X$, we have 
\[
\Ee_{|L} \cong \oplus _{i=1}^r{\mathcal O}_{L}(a_i^L) \quad \text{ with } a_1^L\leq \dots \leq a_r^L,
\]
and we call $(a_1^L,  \dots,  a_r^L)\in \ZZ^{\oplus r}$ the \textit{splitting type} of $\Ee$ on $L$. Moreover, in case of $X = \PP^n$, there exists a non-empty open subset $U$ of the Grassmannian of lines $\mathrm{Gr}(2,n+1)$ and an ordered set of integers $(a_1, \dots , a_r)$ such that $\Ee_{|L}\cong  \oplus _{i=1}^r{\mathcal O}_{L}(a_i)$ for all $L\in U$. For every line $L \in U$, the given splitting is called the \emph{generic splitting type} of $\Ee$ and the set $S(\Ee):=\mathrm{Gr}(2, n+1)\setminus U$ is called the set of \emph{jumping lines}. A vector bundle $\Ee$ on $\PP^n$ is said to be {\it uniform} if $S(\Ee)=\emptyset$, so, when all lines have the same splytting type. As an application, recall for example that the study of $S(\Ee)$ can be used to understand the geometry of moduli space of stable vector bundles; see \cite{Barth}. We will adapt, in Definition \ref{def-jumping}, the concept of jumping line for coherent sheaves and, indeed, the studying of their locus will allow us to recover the divisor from the sheaf of logarithmic differentials.

\begin{example}\label{exmm12}
Let $D$ be a smooth conic in $\PP^2$ and apply Lemma \ref{keylem} to a line $L$. If $L$ is tangent to $D$, Lemma \ref{keylem} produces an exact sequence
\[
0\to \Oo_L(1)  \to \left( \Tt_{\PP^2}(-\log D) \right)_{|L} \to \Oo_L\to 0, 
\]
and so we get $\left( \Tt_{\PP^2}(-\log D) \right)_{|L} \cong \Oo_L\oplus \Oo_L(1)$. Similarly, if $L$ intersects $D$ transversally, we get the same splitting. In particular, the bundle $\Tt_{\PP^2}(-\log D)$ has the same splitting over any line, i.e. it is uniform; see \cite{OSS}. By \cite{vdv}, $\Tt_{\PP^2}(-\log D)$ is isomorphic to either $\Oo_{\PP^2}\oplus \Oo_{\PP^2}(1)$ or $\Tt_{\PP^2}(-1)$. From $\mathrm{h}^0(\Tt_{\PP^2})=8$ and the sequence (\ref{sseeq3}), we see that the former is impossible. Hence we have $\Tt_{\PP^2}(-\log D) \cong \Tt_{\PP^2}(-1)$; see \cite{Angelini} for the generalization. 
\end{example}

\begin{example}
For a divisor $D=L_1+\dots + L_m$ of $m$ distinct lines in $\PP^2$, call the maximal number of lines in $D$ passing through a point, the {\it multiplicity} of $D$, and denote it by $m(D)$; for the generic case we have $m(D)=2$. Pick a point $p\in \PP^2$ achieving $m(D)$ and consider a general line $L$ through $p$. Applying Lemma \ref{keylem} to the pair $(D, L)$, we obtain
\[
0\to \Oo_L\left((1-m+m(D)\right) \to \left( \Tt_{\PP^2}(-\log D) \right)_{|L} \to \Oo_L\left(2-m(D)\right) \to 0.
\]
If $m(D)$ is big enough so that we have $2m(D) \ge m$, the sequence splits, i.e., 
\[
\left( \Tt_{\PP^2}(-\log D) \right)_{|L}  \cong  \Oo_L\left((1-m+m(D)\right) \oplus \Oo_L\left(2-m(D)\right). 
\]
On the other hand, we have
\begin{align*}
c_1(\Tt_{\PP^2}(-\log D))& = 3-m,\\
c_2(\Tt_{\PP^2}(-\log D))&=\sum_{x\in \PP^2}\left (s(x)-1\right) +3-2m,
\end{align*}
where $s(x)$ is the number of lines in $D$ through $x$. Assume further that $2m(D)\ge m+1$ so that 
\[
c_1(\Tt_{\PP^2}(-\log D)\otimes \Oo_{\PP^2}(m-1-m(D))) \le 0 \quad \mbox{and} \quad c_2(\Tt_{\PP^2}(-\log D)\otimes \Oo_{\PP^2}(m-1-m(D)))=0.
\]
Then, by \cite[Corollary 2.13]{EF} or \cite[Lemma 2.4]{FV} we obtain
\[
\Tt_{\PP^2}(-\log D)\otimes \Oo_{\PP^2}(m-1-m(D))  \cong \Oo_{\PP^2} \oplus \Oo_{\PP^2}(m+1-2m(D)).
\]
In particular, $\Tt_{\PP^2}(-\log D)$ is a direct sum of two line bundles so that $D$ is a free divisor by definition. For example, consider a divisor $D=L_1+\dots + L_m$ of $m$ distinct lines, all passing through a fixed point $p$, i.e., $m(D)=m$. Since we have $c_2(\Tt_{\PP^2}(-\log D)\otimes \Oo_{\PP^2}(-1))=0$, we get that $D$ is free with $\Tt_{\PP^2}(-\log D) \cong \Oo_{\PP^2}(1)\oplus \Oo_{\PP^2}(2-m)$. As the second example, for fixed two distinct points $p,q$ in $\PP^2$, consider a divisor 
\[
D=\overline{pq}+ (L_1+ \dots+ L_a)+ (L_1'+ \dots+ L_b')
\]
of distinct $m=a+b+1$ lines in $\PP^2$ such that $\cap_{i=1}^aL_i =\{p\}$ and $\cap_{j=1}^b L_j'=\{q\}$. From the argument above the divisor $D$ is free. On the other hand, by the same method of applying Lemma \ref{keylem}, we can see that the logarithmic vector bundle associated to the divisor $D^{\circ}=(L_1+ \dots+ L_a)+ (L_1'+ \dots+ L_b')$ is not uniform. Assume now that $2m(D)=m$ and $c_2(\Tt_{\PP^2}(-\log D)\otimes \Oo_{\PP^2}(m(D)-2))=0$. Then by the same argument in the above, we get
\[
\Tt_{\PP^2}(-\log D) \cong \Oo_{\PP^2}(2-m(D)) \oplus \Oo_{\PP^2}(1-m(D)). 
\]
For example, a divisor of $6$ lines with $4$ triple points and $3$ double points is free with splitting type $(-1,-2)$. 
\end{example}

\subsection{Generalized logarithmic sheaves}

Now assume that $X$ is a smooth projective surface and $D=\sum_{j=1}^m D_j \subset X$ is a divisor with each $D_j$ irreducible. Fix $k$ distinct points $Z=\{p_1, \dots, p_k\}\subset D_{\mathrm{sm}}$, with $D_{\mathrm{sm}}$ denoting the smooth locus of the divisor. By reordering the index, we may assume that $Z_j=\{p_{a_j}, \dots, p_{a_{j+1}-1}\}$ with $a_1\le \dots \le a_k$ and $Z_j$ is contained in $D_j$ so that $Z=Z_1 \sqcup \dots \sqcup Z_m$; if $a_j=a_{j+1}$, we have $Z_j=\emptyset$. Denoting by $\widetilde{X}=\mathrm{Bl}_Z X$ the blow-up of $X$ along $Z$, one can consider the logarithmic vector bundle $\Omega_{\widetilde{X}}^1(\log \widetilde{D})$, where $\widetilde{D}=\widetilde{D}_1+ \dots + \widetilde{D}_m$ with $\widetilde{D}_j$ the strict transform of $D_j$ via the blow-up $\pi : \widetilde{X} \rightarrow X$. Notice that $\widetilde{D}$ is still a divisor with simple normal crossings in $\widetilde{X}$. By applying the push-forward functor $\pi_*$ to the Poincar\'e residue sequence for $(\widetilde{X}, \widetilde{D})$, one obtains
\begin{equation}\label{lexact}
0\to \pi_*\Omega_{\widetilde{X}}^1 \to \pi_*\Omega_{\widetilde{X}}^1(\log \widetilde{D})\to \oplus_{j=1}^m \Oo_{D_j} \to \mathbf{R}^1\pi_*\Omega_{\widetilde{X}}^1 \to  \mathbf{R}^1\pi_*\Omega_{\widetilde{X}}^1(\log \widetilde{D})\to 0.
\end{equation}

\begin{lemma}\label{lemmm2}
Let $E_i=\pi^{-1}(p_i)$ be the exceptional divisor of $\pi$ at $p_i$. 
\begin{itemize}
    \item [(a)] $\Omega_{\widetilde{X}|X}^1 \cong \oplus_{i=1}^k \Oo_{E_i}(2E_i)$. 
    \item [(b)] $\mathbf{R}^i\pi_*\Omega_{\widetilde{X}|X}^1 \cong \oplus_{j=1}^k \mathrm{H}^i(\Oo_{E_j}(2E_j))\otimes \Oo_{p_j}$, for $i \geq 0$. 
\end{itemize}
\end{lemma}

\begin{proof}
The sequence (iv) of \cite[Lemma 15.4]{Fulton} in our setting is as follows:
\begin{equation}\label{werpo}
0\to \Tt_{\widetilde{X}} \to \pi^* \Tt_X \to \mathcal{Q} \to 0
\end{equation}
for the sheaf $\mathcal{Q}$ fitting into the sequence (i) of \cite[Lemma 15.4]{Fulton}
\[
0\to \Oo_E(-1) \to \Oo_E^{\oplus 2} \to \mathcal{Q} \to 0
\]
with $E=E_1+\dots + E_k$. Thus we have $\mathcal{Q} \cong \Oo_E(1)$. Dualizing the sequence (\ref{werpo}), we get 
\begin{equation}\label{eqa2}
0\to \pi^*\Omega_X^1 \to \Omega_{\widetilde{X}}^1 \to \Omega_{\widetilde{X}|X}^1 \to 0,
\end{equation}
where $\Omega_{\widetilde{X}|X}^1 \cong \mathcal{E}xt_{\widetilde{X}}^1(\mathcal{Q}, \Oo_{\widetilde{X}}) \cong \Oo_E(-2)\cong \oplus_{i=1}^k \Oo_{E_i}(2E_i)$, proving the assertion (a). The assertion (b) follows directly from the definition of higher direct image $\mathbf{R}^i\pi_*(-)$. 
\end{proof}

\begin{proposition}\label{prop1}
The push-forward $\pi_*\Omega_{\widetilde{X}}^1(\log \widetilde{D})$ fits into an exact sequence
\begin{equation}\label{eqa22}
0\to \Omega_X^1 \to \pi_*\Omega_{\widetilde{X}}^1(\log \widetilde{D}) \to \oplus_{j=1}^m \Oo_{D_j}(-Z_j) \to 0.
\end{equation}
\end{proposition}

\begin{proof}
Applying the push-forward functor $\pi_*$ to (\ref{eqa2}),
together with $\pi_*\Oo_{\widetilde{X}} \cong \Oo_X$ and $\mathbf{R}^i\pi_* \Oo_{\widetilde{X}}=0$ for $i\ge 1$, we get an exact sequence
\begin{equation}\label{eqa3}
0\to \Omega_X^1 \to \pi_*\Omega_{\widetilde{X}}^1 \to \pi_*\Omega_{\widetilde{X}|X}^1 \to 0,
\end{equation}
and $\mathbf{R}^i \pi_*\Omega_{\widetilde{X}}^1 \cong \mathbf{R}^i  \pi_*\Omega_{\widetilde{X}|X}^1$ for $i\ge 1$. Then by Lemma \ref{lemmm2} we get $\pi_*\Omega_{\widetilde{X}}^1 \cong \Omega_X^1$ and $\mathbf{R}^1\pi_* \Omega_{\widetilde{X}}^1 \cong \mathbf{R}^1\pi_* \Omega_{\widetilde{X}|X}^1\cong \oplus_{i=1}^k \Oo_{p_i}=\Oo_Z$. 

Set $E=E_i$ and $p=p_i$ for some $i$. By tensoring the sequence (\ref{eqa2}) by $\Oo_E$, we obtain
\[
0\to Tor_{\widetilde{X}}^1(\Oo_E(2E), \Oo_E) \cong \Oo_E(-1) \to \left(\pi^*\Omega_X^1\right)_{|E} \to \left(\Omega_{\widetilde{X}}^1\right)_{|E} \to \Oo_E(-2) \to 0.
\]
Let us observe that, since $\left(\pi^*\Omega_X^1\right)_{|E}\cong\Oo_E^{\oplus 2}$, it turns out that $\left(\Omega_{\widetilde{X}}^1\right)_{|E}\cong \Oo_E(-2)\oplus \Oo_E(1)$ (which is consistent with the fact that $c_1\left(\left(\Omega_{\widetilde{X}}^1\right)_{|E}\right)=(\pi^*K_X+E).E=-1$). Now, given that $\widetilde{D}$ and $E$ intersect transversally at a single point $q\in E$, we obtain the following commutative diagram
\begin{equation}\label{keydiag1}
\begin{tikzcd}[
 column sep=small,row sep=small,
  ar symbol/.style = {draw=none,"\textstyle#1" description,sloped},
  isomorphic/.style = {ar symbol={\cong}},
  ]
 &0   \ar[d] &0 \ar[d]  &0 \ar[d]\\
0\ar[r] &  \Tt_{E}(-\log \{q\})\cong \Oo_{\PP^1}(1) \ar[r]\ar[d, labels=left]& \left(\Tt_{\widetilde{X}}(-\log \widetilde{D})\right)_{|E} \ar[r] \ar[d] & \mathcal{N}_{\log \widetilde{D}/\{q\}} \ar[r]\ar[d,isomorphic] & 0\\
0\ar[r] & \Tt_{E}\cong \Oo_{\PP^1}(2)  \ar[r]\ar[d] &\left(\Tt_{\widetilde{X}}\right)_{|E}  \ar[d] \ar[r] &\mathcal{N}_{E|\widetilde{X}}  \ar[r]& 0\\
 &\Oo_q\ar[r, isomorphic] \ar[d]&  \Oo_q \ar[d] \\
&0  &0 . 
\end{tikzcd}
\end{equation}

From the middle vertical exact sequence, we obtain 
\[
\left(\Omega_{\widetilde{X}}^1(\log \widetilde{D})\right)_{|E} \cong \Oo_{\PP^1}(-1)\oplus \Oo_{\PP^1}(1).
\]
In particular, we have $\mathrm{H}^1\left(\Omega_{\widetilde{X}}^1(\log \widetilde{D})_{|E}\right)=0$. Now by the Theorem on Formal Functions in \cite[Section III.11]{Hartshorne} we have
\[
\left ( \mathbf{R}^1\pi_* \Omega_{\widetilde{X}}^1(\log \widetilde{D}) \right)_p^{\wedge} \cong \lim_{\longleftarrow} \mathrm{H}^1 \left(E^{(n)}, \left( \Omega_{\widetilde{X}}^1(\log {\widetilde{D}})\right)^{(n)}\right),
\]
where $E^{(n)}=\widetilde{X}\times_{X}\mathrm{Spec}(\Oo_p/\mathbf{m}_p^n)$, with $\mathbf{m}_p$ the maximal ideal of the local ring, denotes the thickening of $E$ with order $n$, admitting the exact sequence
\begin{equation}\label{eqa6}
0\to \Oo_{E}(n) \to \Oo_{E^{(n+1)}} \to \Oo_{E^{(n)}} \to 0,
\end{equation}
and $\left( \Omega_{\widetilde{X}}^1(\log \widetilde{D})\right)^{(n)}$ is the restriction to $E^{(n)}$; refer to \cite[proof of Proposition V.3.4]{Hartshorne}. The long exact sequence of cohomology associated to the sequence (\ref{eqa6}) tensored with $\Omega_{\widetilde{X}}^1(\log \widetilde{D})$ is as follows: 
\[
0\cong \mathrm{H}^1\left(\left( \Omega_{\widetilde{X}}^1(\log \widetilde{D})\right)_{|E}(n)\right) \to \mathrm{H}^1\left(\left( \Omega_{\widetilde{X}}^1(\log \widetilde{D})\right)^{(n+1)}\right) \to \mathrm{H}^1\left(\left( \Omega_{\widetilde{X}}^1(\log \widetilde{D})\right)^{(n)}\right)\to 0.
\]
Thus we get that $ \left ( \mathbf{R}^1\pi_* \Omega_{\widetilde{X}}^1(\log \widetilde{D}) \right)_p^{\wedge} $ is trivial. 
\end{proof}

It turns out that the coherent sheaf $\pi_*\Omega_{\widetilde{X}}^1(\log \widetilde{D})$ introduced in Proposition \ref{prop1} will be the main object of study in this paper.

\begin{definition}\label{def-genlog}
We will call the push-forward $\pi_*\Omega_{\widetilde{X}}^1(\log \widetilde{D})$ the {\it generalized logarithmic sheaf} on $X$ associated to the pair $(D, Z)$, and we will denote it by $\Omega_X^1(\log (D, Z))$. The sequence (\ref{eqa2}) is rewritten as 
\begin{equation}\label{poin}
0\to \Omega_X^1 \to \Omega_X^1(\log (D, Z)) \to \oplus_{j=1}^m \Oo_{D_j}(-Z_j) \to 0, 
\end{equation}
and it is called the {\it Poincar\'e residue sequence} for $\Omega_X^1(\log (D, Z))$. When $Z=\emptyset$, one gets the usual logarithmic vector bundle $\Omega_X^1(\log (D, \emptyset))=\Omega_X^1(\log D)$. 
\end{definition}

Notice that $\Omega_X^1(\log (D, Z))$ is defined as the push-forward of a torsion-free $\Oo_{\widetilde{X}}$-sheaf with respect to a surjective map, hence, it is a torsion-free $\Oo_X$-sheaf as well. The goal of the next result is to exactly determine its singular locus. 

\begin{proposition}\label{prop234}
 The generalized logarithmic sheaf  $\Omega_X^1(\log (D, Z))$ is a torsion-free sheaf which is not locally free exactly along $Z$. Furthermore, the canonical injection into its double is described by the following short exact sequence
\begin{equation}\label{dd1}
0\to \Omega_X^1(\log (D, Z)) \to  \Omega_X^1(\log D) \to \Oo_{Z} \to 0.
\end{equation}
In particular, we have that $\left(\Omega_X^1(\log (D, Z))\right)^{\vee\vee} \simeq \Omega_X^1(\log D)$.
\end{proposition}

\begin{proof}
Using the notation introduced at the beginning of the section, set $\widetilde{D}'=\widetilde{D} + (E_1+ \dots + E_k)$. Then by \cite[Proposition 3.2 in Chapter II]{Del} one obtains an injection 
\begin{equation}\label{injection-pullback}
\pi^*\Omega_X^1(\log D) \to \Omega_{\widetilde{X}}^1(\log \widetilde{D}')
\end{equation}
because $\pi^*\Omega_X^1(\log D)$ is locally free. Denote its cokernel by $\Kk$ and set $\widetilde{Z}=\{q_1, \dots, q_k\}$ where $q_i= \widetilde{D} \cap E_i$ are the intersection points. Then we obtain the following commutative diagram
\begin{equation}\label{ddia}
\begin{tikzcd}[
  row sep=small,
  ar symbol/.style = {draw=none,"\textstyle#1" description,sloped},
  isomorphic/.style = {ar symbol={\cong}},
  ]
&0\ar[d]&0\ar[d]&0 \ar[d]   \\
0\ar[r] &\pi^*\Omega_{X}^1 \ar[r] \ar[d] &\pi^*\Omega_X^1(\log D) \ar[d]\ar[r] & \oplus_{j=1}^m \Oo_{\widetilde{D}_j + \pi^{-1}(Z_j)} \ar[r]\ar[d,"\eta_1"] &0\\
0\ar[r] & \Omega_{\widetilde{X}}^1 \ar[r] \ar[d] & \Omega_{\widetilde{X}}^1(\log \widetilde{D}') \ar[r] \ar[d] &\oplus_{j=1}^m\left( \Oo_{\widetilde{D}_j}\oplus (\oplus_{p_i \in Z_j} \Oo_{E_i})\right) \ar[r] \ar[d]& 0\\
0\ar[r]&\oplus_{i=1}^k \Oo_{E_i}(-2) \ar[r] \ar[d] &\Kk  \ar[r] \ar[d] &\Oo_{\widetilde{Z}} \ar[r] \ar[d] & 0\\
&0&0&0.
\end{tikzcd}
\end{equation}
Note that the first horizontal sequence is the pull-back of the Poincar\'e residue sequence for $\Omega_X^1(\log D)$ and the second horizontal sequence is the Poincar\'e residue sequence for $\Omega_{\widetilde{X}}^1(\log \widetilde{D}')$. The middle terms of these two sequences are related by the injection given by (\ref{injection-pullback}). Moreover, by its definition, such an injection factorizes through $\pi^*\Omega_{X}^1 \hookrightarrow \Omega_{\widetilde{X}}^1$ and therefore we obtain the commutative upper left square. We complete the diagram by the Snake Lemma. Finally, from the lower horizontal sequence, we obtain that $\Kk \cong \oplus_{i=1}^k \Oo_{E_i}(-1)$.

Now, applying the push-forward functor $\pi_*$ to the middle vertical sequence, we obtain
\[
\Omega_X^1(\log D) \cong \pi_* \Omega_{\widetilde{X}}^1(\log \widetilde{D}').
\]
Again, take the push-forward functor $\pi_*$ to the following exact sequence
\[
0\to \Omega_{\widetilde{X}}^1(\log \widetilde{D}) \to \Omega_{\widetilde{X}}^1(\log \widetilde{D}') \to \oplus_{i=1}^k \Oo_{E_i} \to 0
\]
to obtain an injection from $\pi_* \Omega_{\widetilde{X}}^1(\log \widetilde{D}) $ into $\Omega_X^1(\log D) $, whose cokernel $\Kk'$ is a subsheaf of $\oplus_{i=1}^k \pi_*\Oo_{E_i} \cong \Oo_Z$. From the following commutative diagram 
\begin{equation}\label{ddia1}
\begin{tikzcd}[
  row sep=small,
  ar symbol/.style = {draw=none,"\textstyle#1" description,sloped},
  isomorphic/.style = {ar symbol={\cong}},
  ]
&0\ar[d]&0\ar[d]   \\
 &\Omega_{X}^1 \ar[r,isomorphic] \ar[d] &\Omega_{X}^1\ar[d] \\
0\ar[r] & \pi_*\Omega_{\widetilde{X}}^1(\log \widetilde{D}) \ar[r] \ar[d] & \Omega_X^1(\log D)\ar[r] \ar[d] &\Kk '\ar[r] \ar[d,isomorphic]& 0\\
0\ar[r]&\oplus_{j=1}^m \Oo_{D_j}(-Z_j) \ar[r] \ar[d] &\oplus_{j=1}^m \Oo_{D_j}  \ar[r] \ar[d] &\Oo_Z \ar[r] & 0\\
&0&0,
\end{tikzcd}
\end{equation}
we get that $\Kk'$ is isomorphic to $\Oo_Z$ and the assertion follows. 
\end{proof}

\begin{remark}
Recall that the sections of $\Omega_X^1(\log D)$ around a point $x\in X$ are meromorphic $1$-forms of the form $\omega+u \left(\frac{d z_1}{z_1}\right)$, where $z_1$ is the local defining equation for $D$ in the local coordinate system $\langle z_1, z_2 \rangle$. We see from Proposition \ref{prop234} that the sections of $\Omega_X^1(\log (D, Z))$ around a point $x\in Z$ are meromorphic $1$-forms of the same form as above with the additional condition that the holomorphic function $u$ vanishes at $x$. 
\end{remark}

%%%%%%%%%%%%%%%%%%%%%%%%

\section{The Torelli property for generalized logarithmic sheaves}\label{toreliness}

In this section we are going to set up the main problem we are interested in, namely, up to which point the generalized logarithmic sheaf $\Omega^1_X(\log(D,Z))$ defined in the previous section determines the pair $(D,Z)$ . By analogy with other well-studied settings, we are going to call it the {\it Torelli problem} for generalized logarithmic sheaves.

\begin{question}\label{question-torelli}
Let $X$ be a projective surface, and $Z \subset X$ a finite set (possibly empty) of distinct points. For an effective line bundle $\mathcal{L}$, set $\FF(Z,\mathcal{L})\subset \PP \mathrm{H}^0(\Ll\otimes \Ii_{Z, X})$ consisting of smooth curves $D$ from the linear system $|\mathcal{L}|$ passing through $Z$. Then the question we want to address is:

\vspace{.2cm}
\begin{nscenter}
"Given two divisors $D_1,D_2\in \FF(Z,\mathcal{L})$ with $\Omega_{X}^1(\log (D_1, Z)) \cong \Omega_{X}^1(\log (D_2, Z))$,\\ do we have $D_1=D_2$?"
\end{nscenter}

\vspace{.2cm}
\noindent If the answer is positive, then we say that {\it the Torelli property holds} for (the divisors in) $\FF(Z,\mathcal{L})$.

\end{question}

First of all, in order to understand the dependency of the Torelli property with respect to the choice of the general set of points $Z\subset X$, let $Z'\subset X$ be a second set of distinct points and let $Z''=Z\cup Z'$. Let us define $\widetilde{X}:=\mathrm{Bl}_{Z}X$ and $\widetilde{X}':=\mathrm{Bl}_{Z''} X$. Since there exists an isomorphism between $\widetilde{X}\backslash\pi^{-1}(Z)$ and $X\backslash Z$, we are going to identify the points $Z'$ in $X$ with their preimages in $\widetilde{X}$. Let us consider the composition of blow-up morphisms:

\begin{equation}\label{abcd}
\pi{''}:\widetilde{X}'\stackrel{\pi'}{\longrightarrow}\widetilde{X}\stackrel{\pi}{\longrightarrow} X
\end{equation}

\noindent Now we have two possibilities to study:

\subsection{Divisors $D\subset X$ that contain the larger set of points $Z''$}\label{subsec-pointsin}

In this case we are considering divisors $D$ from the linear system $|\mathcal{L}|$ in $\FF(Z'',\mathcal{L})\subset \FF(Z,\mathcal{L})$. Therefore, in order to have a consistent definition for the Torelli property, we have to prove the following.

\begin{proposition}\label{prop-Torelliplusone}
For two given divisors $D_1,D_2\in\FF(Z{''},\mathcal{L})$ with
\[
\Omega_{X}^1(\log (D_1, Z{''})) \cong \Omega_{X}^1(\log (D_2, Z{''})),
\]
we have
\[
\Omega_{X}^1(\log (D_1, Z)) \cong \Omega_{X}^1(\log (D_2, Z))
\]
\end{proposition}

\begin{proof}
In $\widetilde{X}$ we can consider the aforementioned short exact sequence
\[
0\to \Omega_{\widetilde{X}}^1(\log (\widetilde{D},Z')\to \Omega_{\widetilde{X}}^1(\log \widetilde{D})\to\Oo_{Z'}\to 0
\]
for $D=D_i$. Taking push-forward with respect to $\pi$, we have
\[
0\to \Omega_{X}^1(\log (D,Z''))\to \Omega_{X}^1(\log (D,Z))\to \Oo_{Z'}\to 0
\]
from the fact that $\mathbf{R}^1\pi_*\Omega_{\widetilde{X}}^1(\log (\widetilde{D},Z'))\cong 0$ since $\Omega_{\widetilde{X}}^1(\log \widetilde{D},Z')$ is locally free along $Z$. If we let $\phi:\Omega_{X}^1(\log (D_1, Z'')) \stackrel{\cong}{\longrightarrow}\Omega_{X}^1(\log (D_2, Z''))$ 
be an isomorphism, then we can complete the following diagram 
\[
\begin{tikzcd}[
 column sep=normal,row sep=normal,
  ar symbol/.style = {draw=none,"\textstyle#1" description,sloped},
  isomorphic/.style = {ar symbol={\cong}},
  ]
0\ar[r] & \Omega_{X}^1(\log (D_1, Z'')) \ar[r]\ar[d,"\phi"] & \Omega_{X}^1(\log (D_1, Z)) \ar[r]\ar[d,dotted,"\psi"]  & \Oo_{Z'} \ar[r]\ar[d,isomorphic] & 0\\
0\ar[r] & \Omega_{X}^1(\log (D_2, Z'')) \ar[r] &  \Omega_{X}^1(\log (D_2, Z)) \ar[r] & \Oo_{Z'} \ar[r] & 0
\end{tikzcd}
\]
as soon as we prove that $\Ext_X^1(\Oo_{Z'},\Omega_{X}^1(\log (D_2, Z)))$ is trivial. This follows easily from the local-to-global spectral sequence taking into account that the intersection of the support of $\Oo_{Z'}$ and the singular points of $\Omega_{X}^1(\log (D_2, Z))$ is empty. Therefore, we can construct an isomorphism $\psi:\Omega_{X}^1(\log (D_1, Z)) \stackrel{\cong}{\longrightarrow}\Omega_{X}^1(\log (D_2, Z))$, as required. 
\end{proof}

\begin{remark}
Notice that the previous proposition was already known in the particular case $Z=\emptyset$ by Proposition \ref{prop234}.
\end{remark}

\begin{corollary}
Let $Z\subset Z{''}\subset X$ general set of points. If the Torelli property holds for $\FF(Z,\mathcal{L})$, then it also holds for $\FF(Z{''},\mathcal{L})$.
\end{corollary}

    The previous observations drive us to consider the following problem: given a projective surface $X$ and an effective line bundle $\mathcal{L}$ on $X$, which is the minimal cardinality of a general set of points $Z\subset X$ such that the Torelli property holds for $\FF(Z,\mathcal{L})$? In Section \ref{sec-projplane} we are going to address this question in the case of arbitrary line bundles on the projective plane $\PP^2$.

\subsection{Divisors $D\subset X$ such that $D\cap Z'=\emptyset$}\label{subsec-pointsout}

Let $D\subset X$ be a divisor such that $Z\subset D_{\mathrm{sm}}$ and $Z'\cap D=\emptyset$. Let $\widetilde{D}$ (resp. $\widetilde{D}'$) be the strict transform of $D$ in $\widetilde{X}$ (resp. in $\widetilde{X}'$), according to notation from (\ref{abcd}).

\begin{remark}\label{composition-pushforward}
It follows from Proposition \ref{prop1} that $\pi{'}_*\Omega_{\widetilde{X}'}^1(\log\widetilde{D}')\cong\Omega_{\widetilde{X}}^1(\log \widetilde{D})$.
In other words, we have the following isomorphism 
\[
\Omega_X^1(\log (D, Z)) \cong (\pi \circ \pi')_* \Omega^1_{\widetilde{X}'} (\log \widetilde{D}'),
\]
where $\widetilde{D}'$ is the strict transform of $\widetilde{D}$ in $\widetilde{X}'$. 
\end{remark}

This allows us to study  logarithmic bundles in more general blown-up surfaces, a problem that will be tackled in Section \ref{sec-blowup}. In particular, we would like to check if the Torelli property holds for the linear system $|\widetilde{D}'|$ on $\widetilde{X}'$.  As an immediate consequence of Remark \ref{composition-pushforward}, we have the following result.

\begin{lemma}\label{torr-proj}
If the Torelli property holds for the pair $(D,Z)$ on $X$, then the Torelli property holds for the linear system $|\widetilde{D}'|$ on $\widetilde{X}'$.
\end{lemma}

On the other hand, the reciprocal implication does not hold in general, as we will see combining the previous Lemma with Remark \ref{rem998}. Indeed, in Section \ref{sec-blowup} we are going to study deeply this issues in the particular case of the smooth cubic surface $S\subset\PP^3$.

%%%%%%%%%%%%%%%%%%%%%%%%%%%%%%%%%%%%%%

\section{Generalized logarithmic sheaves with fixed  points inside the divisor}\label{sec-projplane}

In this section we will place ourselves on the setting of Subsection \ref{subsec-pointsin} for the particular case of the projective plane. In other words, we are going to consider the generalized logarithmic sheaf $\Omega_{\PP^2}^1(\log (D ,Z))$ on $\PP^2$ associated to a pair $(D, Z)$, where $D\subset \PP^2$ is a smooth curve of degree $d$ and $Z$ is a set of $k$ distinct points on $D$. In Section \ref{sec-blowup}, in order to follow the path proposed in Subsection \ref{subsec-pointsout}, we are going to study an analogous setting but with respect to a set of point without intersection with the curves under consideration.

Fix $k$ distinct points $Z=\{p_1, \dots, p_k\}$ on $\PP^2$ in general position and set $\FF_d(Z):=\FF(Z,\Oo_{\PP^2}(d))\subset \PP \mathrm{H}^0(\Ii_{Z, \PP^2}(d))$ consisting of the smooth curves of degree $d$ passing through $Z$. We are going to answer to Question \ref{question-torelli} in this setting.

Recall that a hypersurface $D\subset\PP^n$ is of Sebastiani-Thom type if its defining polynomial can be written, after choosing properly a coordinate system, as a sum of two polynomials on disjoint sets of variables. For instance, any hyperplane is of Sebastiani-Thom type. Analogously, any quadric hypersurface in the projective space is of Sebastiani-Thom type, since its defining equation can be written as a sum of squares.\\
If $d\ge 3$  we already know, by \cite[Theorem 1]{UY}, that the Torelli  property holds for any smooth hypersurface not of Sebastiani-Thom type. Focusing on the projective plane, this implies, in particular, that for any $d\geq 3$ and $Z\subset\PP^2$, the Torelli property holds for any pair $(D,Z)$, for  $D$ in the open subset $U\subset\FF_d(Z)$ consisting on smooth curves not of Sebastiani-Thom type.  Indeed, if for two $D_1,D_2$ smooth curves not of Sebastiani-Thom type, the associated generalized logarithmic sheaves are isomorphic,  
$$
\Omega_{\PP^2}^1(\log (D_1, Z_1)) \cong \Omega_{\PP^2}^1(\log (D_2, Z_2)),
$$ then, taking their double duals, we obtain from Proposition \ref{prop234} that
\[
\Omega_{\PP^2}^1(\log D_1) \cong \Omega_{\PP^2}^1(\log D_2).
\]
The mentioned result \cite[Theorem 1]{UY} implies that the curves $D_1$ and $D_2$ coincide. Furthermore, we also have that $Z_1 = Z_2$ because these sets are the respectively singular locus of isomorphic sheaves (see Proposition \ref{prop234}). \\ 

Let us denote the different cases to be studied by the degree of the curve and the number of fixed points, i.e., by the pair $(d,k)$. Notice that, for Question \ref{question-torelli} not to be trivial, we need to have at least two different curves in $\FF_d(Z)$. Therefore, the pair $(d,k)$ should be chosen to satisfy $\mathrm{h}^0(\Ii_{Z, \PP^2}(d))\ge 2$ so that the number $k$ of fixed points does not determine completely the curve. Thus, we will first consider Question \ref{question-torelli} for
\begin{equation}\label{list-Torellicases}
(d,k) \in \{(1,0),(1,1), (2,0),(2,1), (2,2), (2,3), (2,4)\}.
\end{equation}
The three first cases on the list are easily handled. Indeed, recall that we have
\[
\Omega_{\PP^2}^1(\log L) \cong \Oo_{\PP^2}(-1)^{\oplus 2},~\quad \Omega_{\PP^2}^1(\log Q) \cong \Tt_{\PP^2}(-2),
\]
for any line $L$ and any smooth conic $Q$ in $\PP^2$; see, for example, \cite{DK} and \cite{Angelini}, or apply Lemma \ref{keylem}. Moreover, if $p\in L$ is a point on a line $L\subset \PP^2$, then the sequence (\ref{dd1}) gives
\[
\Omega_{\PP^2}^1(\log (L, \{p\})) \cong \Ii_{\{p\}, \PP^2}(-1)\oplus \Oo_{\PP^2}(-1).
\]
Therefore, the Torelli property does not hold for $\FF_1(\{p\})$.\\
In Section \ref{sec-planeconics}, we will deal with the rest of the cases listed in (\ref{list-Torellicases}), that means we will consider a conic and fix at least one point on it.\\
Finally, in Section \ref{sec-planecubics}, we will consider cubic curves of Sebastiani-Thom type.

\subsection{Conics with fixed points}\label{sec-planeconics}
We now consider the case of generalized logarithmic sheaves associated to conics. In this case, an important tool to study them will be the description of their restriction to lines. Before Example \ref{exmm12} we recalled the definition of jumping line for a vector bundle on a projective space. This notion can be generalized for arbitrary coherent sheaves. In the particular case of rank two coherent sheaves, we give the definition as follows. 

\begin{definition}\label{def-jumping}
For a coherent sheaf $\Ee$ on $\PP^2$ of rank two with first Chern class $c_1\in \{-1,0\}$, a line $L \subset \PP^2$ is said to be a \textit{jumping line} of $\Ee$ if $\mathrm{h}^1(\Ee(-1-c_1)_{|L})>0$. Again we denote the set of jumping lines of $\Ee$ by $S(\Ee)$.
\end{definition}

\begin{remark}
Notice that, when $\Ee$ is a vector bundle of rank two with Chern class $c_1=-1$ (resp. $c_1=0$), $L$ is a jumping line if and only if $\Ee_{|L}\ncong\Oo_{\PP^1}(-1)\oplus\Oo_{\PP^1}$ (resp. 
$\Ee_{|L}\ncong\Oo_{\PP^1}\oplus\Oo_{\PP^1}$),  namely $\Ee_{|L}$ has not the generic expected splitting type.
\end{remark}

We will see that the case $(d,k)=(2,3)$ plays a particular role. Let us start to describe it in the following remark.

\begin{remark}\label{keyrem}
Fix a smooth conic $Q$ passing through three points $Z=\{p_1, p_2, p_3\}$ in general position. Then, directly from Definition \ref{def-genlog}, the  generalized logarithmic sheaf $\Omega_{\PP^2}^1(\log (Q, Z))$ fits into the following exact sequence
\begin{equation}\label{eeeeqq2}
0\to  \Omega_{\PP^2}^1 \to \Omega_{\PP^2}^1(\log (Q, Z)) \to \Oo_Q(-Z) \to 0, 
\end{equation}
where $\Oo_Q(-Z)$ is the ideal sheaf of $Z$ inside $Q$. Recall that the inclusion $Z\subset Q$ induces a reversed inclusion of the respective ideals, which implies the following short exact sequence
\[
0\to \Ii_{Q, \PP^2} \to \Ii_{Z, \PP^2}\to \Oo_Q(-Z)\to 0.
\]
Having that $\Ii_{Q, \PP^2} \simeq \Oo_{\PP^2}(-2)$ and $\Ii_{Z, \PP^2}$ admits a free resolution
\[
0\to \Oo_{\PP^2}(-3)^{\oplus 2} \to \Oo_{\PP^2}(-2)^{\oplus 3} \to \Ii_{Z, \PP^2} \to 0,
\]
we obtain, through the Mapping Cone (see \cite[Section 1.5.1]{weibel}), a  free resolution for $\Oo_Q(-Z)$: 
\[
0\to \Oo_{\PP^2}(-3)^{\oplus 2} \to \Oo_{\PP^2}(-2)^{\oplus 2} \to \Oo_{Q}(-Z) \to 0.
\]

Then, considering the resolution $\Omega_{\PP^2}^1$ and applying the Horseshoe Lemma (see \cite[Lemma 2.2.8]{weibel}) to the sequence (\ref{eeeeqq2}), we obtain the exact sequence
\begin{equation}\label{eqa7}
0\to \Oo_{\PP^2}(-3)^{\oplus 3} \stackrel{\mathrm{M}}{\longrightarrow} \Oo_{\PP^2}(-2)^{\oplus 5} \to \Omega_{\PP^2}^1(\log (Q,Z)) \to 0,
\end{equation}
that is, a Steiner type resolution; see for example \cite[Definition 3.1]{DK}. Notice in principle that the Horseshoe Lemma would apply only to the associated modules of twisted global sections. However, since the three sheaves in (\ref{eeeeqq2}) have no global sections and $\mathrm{H}^1(\Omega_{\PP^2}^1(t))=0$ for $t>0$, the Horseshoe Lemma can be also applied to (\ref{eeeeqq2}). In particular, $\Omega_{\PP^2}^1(\log (Q, Z))$ is a stable sheaf of rank two on $\PP^2$ with the Chern classes $(c_1, c_2)=(-1,4)$.
\end{remark}

\begin{proposition}\label{unique-singular-extension}
Fix a set $Z=\{p_1, \dots, p_k\}\subset \PP^2$ of $k$ distinct points on a smooth conic $Q\subset \PP^2$. Among the extensions 
\begin{equation}\label{ert4}
0\to \Omega_{\PP^2}^1 \to  \Ff \to \Oo_Q(-Z) \to 0,
\end{equation}
there exists a unique middle term $\Ff$ exactly singular along $Z$, namely $\Ff \cong \Omega_{\PP^2}^1(\log (Q,Z)).$
\end{proposition}
\begin{proof} 

Let us assume that the sheaf $\Ff$ defined by the sequence (\ref{ert4}) is singular along $Z$. Since $\Oo_Q(-Z)$ is supported on a proper subvariety of $\PP^2$, we have that $\Hom_{\PP^2}(\Oo_Q(-Z),\Oo_{\PP^2})=0$ and then by dualizing the exact sequence given in (\ref{ert4}) we get 
\[
0\to \Ff^{\vee} \to \Tt_{\PP^2} \to \Oo_Q(Z)\otimes \Oo_{\PP^2}(2) \to  \mathcal{E}xt_{\PP^2}^1(\Ff,\Oo_{\PP^2})\cong \Oo_Z \to 0.
\]
The latter sequence decomposes into the following two short exact sequences
\[
 \left\{
                                           \begin{array}{ll}
&0\to \Ff^{\vee} \to \Tt_{\PP^2} \to \Oo_Q\otimes \Oo_{\PP^2}(2) \to 0; \\[1em]
&0\to \Oo_Q\otimes \Oo_{\PP^2}(2) \to \Oo_Q(Z)\otimes \Oo_{\PP^2}(2) \to \Oo_Z \to 0.
\end{array}
\right.
\]
The first exact sequence is unique. Indeed, notice that every surjection of type
\[
\Tt_{\PP^2} \to  \Oo_Q(2)
\]
factorizes through the restriction of the tangent bundle on the conic $Q$, i.e., it can be described by the composition
\[
\Tt_{\PP^2} \stackrel{\pi}{\longrightarrow} \Tt_{\PP^2} \otimes \Oo_Q\cong \Oo_Q(\{q_1,q_2,q_3\})^{\oplus 2} \stackrel{f}{\longrightarrow} \Oo_Q(2),
\]
with $\{q_1,q_2,q_3\}$ three points on the conic. Observe that, by the isomorphism $Q \cong \PP^1$, the second map can be rewritten as
\[
\Oo_{\PP^1}(3)^{\oplus 2} \stackrel{f}{\longrightarrow} \Oo_{\PP^1}(4) \to 0.
\]
This implies that $f$ is unique, up to isomorphism; in particular, it can be represented by two linearly independent linear forms that give a basis of $\mathrm{H}^0\left(\Oo_{\PP^1}(1)\right)$. This implies that, again up to isomorphism, there is a unique vector bundle that is the kernel of $(f \circ \pi)$. Specifically, it is the vector bundle $\Tt_{\PP^2} (-2)$. 

The second exact sequence is uniquely determined by the polynomial of degree $k$ on $Q$ vanishing along $Z$, which represents the injective map of the sequence.\\

Putting everything together, this implies the uniqueness of the sheaf $\Ff$, because we can recover $\Ff$ as the kernel of the composition
\[
\Ff^{\vee\vee} \to \mathcal{E}xt_{\PP^2}^1(\Oo_Q\otimes \Oo_{\PP^2}(2), \Oo_{\PP^2}) \cong \Oo_Q \to \Oo_Z,
\]
which is obtained from the dual of the first sequence and the twist of the second one. 

Let us remark that if a sheaf $\Ff$ in the sequence (\ref{ert4}) is singular along a proper subscheme $Z^{'}\subset Z$ then the surjection 
\[
\Tt_{\PP^2} \to  \Tt_{\PP^2} \otimes \Oo_Q\cong \Oo_Q(\{q_1,q_2,q_3\})^{\oplus 2} \to \Oo_Q(Z\setminus Z^{'})\otimes \Oo_{\PP^2}(2)
\]
is no more unique. 
\end{proof}

The following result will tell us that the jumping lines of $\Omega_{\PP^2}^1(\log (Q,Z))$ are the expected ones, i.e., the ones passing through any two points of $Z$ and the ones passing through  a point of $Z$  and tangent to the curve $Q$.
\begin{proposition}\label{tangent}\label{tangentZ=3}
For $Z=\{p_1, \dots, p_k\}\subset \PP^2$ a set of $k \geq 3$ distinct points in a smooth conic $Q\subset \PP^2$, we have
\[
S\left(\Omega_{\PP^2}^1(\log (Q,Z)\right)= \left\{L_{ij}~|~ 1\le i<j\le k\right\} \cup \left\{ T_iQ ~|~1\le i \le k\right\},
\]
where $L_{ij}$ is the line passing through $p_i$ and $p_j$, and $T_iQ$ is the tangent line of $Q$ at $p_i$. 
\end{proposition}

\begin{proof}
From the sequence (\ref{dd1}) we get that the lines $L_{ij}$'s are jumping. Indeed, restricting (\ref{dd1}) to $L_{ij}$, we get an exact sequence
\[
0\to \Oo_{p_i}\oplus \Oo_{p_j}\to \Omega_{\PP^2}^1(\log (Q,Z))_{|L_{ij}} \to \Oo_{L_{ij}}(-1)\oplus  \Oo_{L_{ij}} \to \Oo_{p_i}\oplus \Oo_{p_j}\to  0.
\]
Denoting by $\Gg$ the image of the middle map 
\[
\Omega_{\PP^2}^1(\log (Q,Z))_{|L_{ij}} \to \Oo_{L_{ij}}(-1)\oplus  \Oo_{L_{ij}},
\]
we verify directly that $\hh^1(\Omega_{\PP^2}^1(\log (Q,Z))_{|L_{ij}})=\hh^1(\Gg)$. This  proves that $\hh^1(\Omega_{\PP^2}^1(\log (Q,Z))_{|L_{ij}})\ge 1.$

Assume that $k=3$ and then, as seen before, we have 
\[
0\to \Omega_{\PP^2}^1 \to \Omega_{\PP^2}^1(\log (Q,Z))\to \Oo_Q(-Z) \to 0
\]
and a resolution determined by a matrix of linear forms, that is a Steiner sheaf:
\[
0\to \Oo_{\PP^2}(-3)^{\oplus 3} \stackrel{\mathrm{M}}\longrightarrow  \Oo_{\PP^2}(-2)^{\oplus 5} \to \Omega_{\PP^2}^1(\log (Q,Z)) \to 0.
\]
According to Proposition \ref{unique-singular-extension}, such a sheaf, singular along $Z$, is unique.

On the other hand, up to a projective automorphism, we may assume $Q=V(x_0x_1+x_1x_2+x_2x_0)$ and consider the matrix 
\[
\mathrm{M}=\begin{pmatrix} x_0&x_1&x_2&0&0\\ ax_0&bx_1&ax_2&x_0&-x_2\\ bx_0&ax_1&ax_2&x_1&x_1+x_2   \end{pmatrix}^t. 
\]
which defines a  Steiner sheaf $\Ff$, that fits as well in the short exact sequences (\ref{eeeeqq2}) and (\ref{eqa7}) replacing $\Omega_{\PP^2}^1(\log (Q, Z))$. 
Then $\Ff$ is a non-locally free sheaf with singularity at three points $p_0=[1:0:0]$, $p_1=[0:1:0]$, $p_2=[0:0:1]$, and $\Ff^{\vee\vee} \cong \Tt_{\PP^2}(-2) \cong \Omega_{\PP^2}^1(\log Q)$.
By unicity this gives $\Ff\cong \Omega_{\PP^2}^1(\log (Q, Z))$.
 Notice that the three tangent lines $V(x_i)$ for $i=0,1,2$ are jumping lines for $\Ff$. Furthermore, we can observe that, for 
\[
u_0=(0,1,1,0,a-b)^t,\quad u_1=(1,0,1,a-b,-a+b)^t,\quad u_2=(1,1,0,b-1,0)^t, 
\]
each $\mathrm{M}u_i$ is divisible by $x_j+x_k$, whenever we have $\{i,j,k\}=\{0,1,2\}$. This implies that the three additional lines $V(x_0+x_1)$, $V(x_1+x_2)$ and $V(x_2+x_0)$ are also jumping lines for $\Ff$; these three additional lines are the tangent lines to $D$ at each $p_i$. Let us explain why in a few words. By restriction of the free resolution of $\Ff$ to one of these six lines, say $L$, we obtain:
\[
0\to \Oo_{L}(-3)^{\oplus 3} \stackrel{\mathrm{M}_L}\longrightarrow  \Oo_{L}(-2)^{\oplus 5} \to \Ff\otimes \Oo_L \to 0.
\]
In the long exact sequence of cohomology we have:
\[
\mathrm{H}^1(\Oo_{L}(-3)^{\oplus 3}) \stackrel{{\mathrm{M}}_L}\longrightarrow  \mathrm{H}^1(\Oo_{\PP^2}(-2)^{\oplus 5}) \to \mathrm{H}^1(\Ff\otimes \Oo_L) \to 0.
\]
The Serre duality provides the isomorphisms 
\[
\mathrm{H}^1(\Oo_{L}(-3)^{\oplus 3})\cong \mathrm{H}^0(\Oo_{L}(1)^{\oplus 3})^{\vee}\mbox{ and } \mathrm{H}^1(\Oo_{L}(-2)^{\oplus 5})\cong \mathrm{H}^0(\Oo_{L}^{\oplus 5})^{\vee},
\]
which identify $\mathrm{H}^1(\Ff\otimes \Oo_L)$ with the kernel of the transpose matrix $(\mathrm{M}_L)^t$. Since each one of the six lines gives a nonzero vector in this kernel, this proves that $\hh^1(\Ff\otimes \Oo_L)\ge 1.$ Let $L$ be a line passing through only one of the points $p_i$ with $i=1,2,3$ and not tangent to $Q$. Consider the following restriction diagram
\[
\xymatrix{
\left(\Omega^1_{\PP^2}\right)_{|L}\cong \Oo_L(-2) \oplus \Oo_L(-1) \ar[r] \ar[dr]_\varphi & \left(\Omega_{\PP^2}^1(\log (Q, Z))\right)_{|L} \ar[d]\\
& \left(\Omega_{\PP^2}^1(\log (Q))\right)_{|L}\cong \Oo_L(-1) \oplus \Oo_L
}
\]
Due to the Key Restriction Lemma \ref{keylem}, we have that $\varphi$ defines an isomorphism between the two $\Oo_L(-1)$ summands. Hence $\left(\Omega_{\PP^2}^1(\log (Q, Z))\right)_{|L} \cong \Oo_L(-1)^{\oplus 2} \oplus \Oo_{p_i}$, i.e., $L$ cannot be a jumping line.

Therefore, we get the assertion for $k=3$. 

Assume now that $Z$ is a set of $k+1$ distinct points, and denote $\Ff_{k+1}$ and $\Ff_k$, the sheaves singular along $Z$ and along a subscheme $Z^{'}\subset Z$ of $k$ points, respectively. Then we have an exact sequence 
\[
0\to \Ff_{k+1} \to  \Ff_k \to \Oo_{Z\setminus Z^{'}} \to 0.
\]
Outside $Z\setminus Z^{'}$, their jumping lines coincide. So, choosing the $\binom{k+1}{k}=k+1$ subschemes $Z^{'}$ of length $k$ in $Z$, we verify that the jumping lines are 
the $\binom{k+1}{2}$ lines joining two points from $Z$ and the tangent lines to $Q$ along $Z$ (the $k$ vertices of the $k$-gon) by inductive hypothesis. This concludes the assertion. \end{proof}

We are ready to state the main result of this section.
\begin{theorem}\label{torr-k}
The Torelli property holds for $\FF_2(Z)$ if and only if $|Z|\ge 3$;
\end{theorem}

\begin{proof}
Note that the case $|Z|\ge 5$ is obvious, since any five points determine at most one smooth conic. Recall that for a set $Z=\{p_1, \dots, p_k\}\subset \PP^2$ of $k$ distinct points in general position and a smooth conic $Q\subset \PP^2$ containing $Z$, 
we have the double dual exact sequence
\begin{equation}\label{eqa1000}
0\to  \Omega_{\PP^2}^1(\log (Q,Z))  \to  \Omega_{\PP^2}^1(\log Q) \cong \Tt_{\PP^2}(-2) \to \Oo_{Z} \to 0,
\end{equation}
fitting into the following commutative diagram

\begin{equation}\label{dddia}
\begin{tikzcd}[
  row sep=small,
  ar symbol/.style = {draw=none,"\textstyle#1" description,sloped},
  isomorphic/.style = {ar symbol={\cong}},
  ]
&&0\ar[d] &0\ar[d]  \\
0\ar[r] & \Omega_{\PP^2}^1\ar[r] \ar[d,isomorphic] & \Omega_{\PP^2}^1(\log (Q,Z)) \ar[d]\ar[r] & \Oo_Q(-Z) \ar[r]\ar[d] &0\\
0\ar[r] & \Omega_{\PP^2}^1 \ar[r]  & \Omega_{\PP^2}^1(\log Q)) \cong \Tt_{\PP^2}(-2) \ar[r] \ar[d] &\Oo_Q \ar[r] \ar[d]& 0\\
& &\Oo_{Z}\ar[r,isomorphic] \ar[d] &\Oo_{Z}\ar[d] & \\
&&0&0.
\end{tikzcd}
\end{equation}
From the middle vertical sequence of (\ref{dddia}) we get that the sheaf $\Omega_{\PP^2}^1(\log (Q, Z))$ is a kernel of a surjection $\Tt_{\PP^2}(-2) \rightarrow \Oo_Z$, which induces a surjection $\left ( \Tt_{\PP^2}(-2) \right)_Z \rightarrow \Oo_Z$ on stalks. Thus we get a morphism 
\[
\FF_2(Z) \longrightarrow  \underbrace{\PP^1 \times \dots \times \PP^1}_{k \text{ copies}}\cong \prod_{i=1}^k \PP \mathrm{Hom}(\Tt_{\PP^2}(-2)_{p_i}, \Oo_{p_i}).
\]
Since $\FF_2(Z)$ is an open Zariski subset of $\PP^{5-k}\cong \PP \mathrm{H}^0(\Ii_{Z, \PP^2}(2))$, by dimension counting the map is not injective for $k=1,2$, i.e., the Torelli property does not hold. 

The assertion holds for $k=3,4$ by Proposition \ref{tangentZ=3}, because $k$ points on a smooth conic, determined by the singular locus of the logarithmic sheaf, together with the $k$ tangent lines at the $k$ points, determined by its jumping lines, fix the conic. \end{proof}

As a direct consequence, we have the following result.
\begin{corollary}
Let $Q$ be a smooth conic in $\FF_2(Z)$ for $Z=\{p_1, p_2, p_3\}$. Then the sheaf       $\Omega_{\PP^2}^1(\log (Q,Z))$ is stable. In particular, we have a generically one-to-one rational map 
\[
\Phi_Z: \PP^2=\PP \mathrm{H}^0(\Ii_{Z,\PP^2}(2)) \dashrightarrow \mathbf{M}_{\PP^2}(-1, 4),
\]
where $\mathbf{M}_{\PP^2}(-1,4)$ is the moduli space of stable sheaves of rank two on $\PP^2$ with Chern classes $(c_1, c_2)=(-1,4)$. 
\end{corollary}

\begin{remark}
The moduli space $\mathbf{M}_{\PP^2}(c_1, c_2)$ of semistable sheaves of rank two on $\PP^2$ with the Chern classes $(c_1, c_2)$ is an irreducible variety of dimension $4c_2-c_1^2-3$, and it is expressed as the following disjoint union
\[
\mathbf{M}_{\PP^2}(c_1, c_2)=\coprod_{\delta\ge 0} \mathbf{M}_{\PP^2}(c_1, c_2)^{\delta},
\]
where $\mathbf{M}_{\PP^2}(c_1, c_2)^{\delta}$ consists of the semistable sheaves $\Ff$ in $\mathbf{M}_{\PP^2}(c_1, c_2)$ with $c_2(\Ff^{\vee\vee})=c_2-\delta$. Notice that $\mathbf{M}_{\PP^2}(c_1, c_2)^0$ consists of the locally free sheaves, and in \cite{Li} the natural map 
\[
\Psi: \mathbf{M}_{\PP^2}(c_1, c_2) \to \coprod_{\delta \ge 0} \mathbf{M}_{\PP^2}(c_1, c_2-\delta)^0\times \mathrm{Sym}^{\delta}(\PP^2)
\]
defined by sending $\Ff$ to a pair $(\Ff^{\vee\vee}, \mathrm{coker}(\Ff \rightarrow \Ff^{\vee\vee}))$ is shown to be a morphism of projective varieties, where the target space is called the {\it Uhlenbeck compactification}. Since the image of the composite $\Psi \circ \Phi_Z$ is a single point $(\Tt_{\PP^2}(-2), Z_{123})$, we obtain a generically one-to-one rational map
\[
\Phi_Z : \PP^2=\PP \mathrm{H}^0(\Ii_{Z,\PP^2}(2)) \dashrightarrow  \Psi^{-1}((\Tt_{\PP^2}(-2), Z_{123})). 
\]
Note that the fibre $\Psi^{-1}((\Tt_{\PP^2}(-2), Z_{123})$, parametrizing the kernels of the surjection $\Tt_{\PP^2}(-2) \rightarrow \Oo_{Z_{123}}$, is isomorphic to $\PP^1\times \PP^1 \times \PP^1$. It might be an interesting question to specify the linear system defining the map $\Phi_Z$. One of the possible ingredients for answering to the question, would be to investigate the logarithmic sheaves associated to singular conics.
\end{remark}

%%%%%%%%%%%%%%%%%%%%%%%%%%%%%%%%%%%%%%%%%%

\subsection{Plane cubic curves}\label{sec-planecubics}
As stated at the beginning of the current section, the question, regarding the Torelli problem, we will focus on Sebastiani-Thom type curves of degree three.
In particular, in this section we will prove that the Torelli property does not hold for the generalized logarithmic sheaf associated to the pair $(D,Z)$ of a cubic curve $D$ of Sebastiani-Thom type and a fixed set of points $Z$, except for the trivial case when $Z$ already determines the curve $D$.  More precisely, we are going to show that for any cubic curve $D$ of Sebastiani-Thom type and any set $Z$ of three aligned inflection points of $D$, the pair $(D,Z)$ belongs to a one-dimensional family of curves with the same generalized logarithmic sheaf. 

Let $D=V(f)$ be a smooth cubic curve in $\PP^2$ and $\Tt_{\PP^2}(-\log D)$ its logarithmic tangent bundle. 
It is well known (see for instance \cite{UY2}) that  $\Tt_{\PP^2}(-\log D)$ is a stable vector bundle of rank two with Chern classes $(c_1, c_2)=(0,3)$ that fits into an exact sequence 
\[
0\to \Oo_{\PP^2}(-3) \to \Oo_{\PP^2}(-1)^{\oplus 3} \to \Tt_{\PP^2}(-\log D) \to 0,
\]
By \cite{Barth} its set of jumping lines $S_D:=S\left((\Tt_{\PP^2}(-\log D))\right)$ is a cubic curve in $(\PP^2)^\vee$. Because of the above resolution we get that $\left( \Tt_{\PP^2}(-\log D)\right)_{|L} \cong \Oo_L(-a)\oplus \Oo_L(a)$ for $a\in \{0,1\}$, i.e., the jumping lines of $\Tt_{\PP^2}(-\log D)$ are all of order $1$. Even if the jumping order is at most 1, the curve of jumping lines is sometimes singular. Let us recall below the behaviour of jumping lines (see also \cite[Proposition 5]{UY2}). We denote by $\langle\nabla f\rangle$ the vector space $\mathrm{Vect}\,(\partial_x f, \partial_y f, \partial_z f)$.

\begin{proposition} Let $L$ be a line in $\PP^2$ and denote by $\ell$ both its corresponding point in $(\PP^2)^{\vee}$ and its defining linear form.
\begin{enumerate}
    \item $l\in S_D$ if and only if it exists a linear form $h$ such that $(\ell \cdot h) \in \langle\nabla f\rangle$.
    \item   $S_D$ is singular at $\ell$  if and only if  
    $(\ell^2) \in \langle\nabla f\rangle$.
    \item If  $S_D$ is singular, then it is a triangle. 
\end{enumerate}
\end{proposition}
\begin{proof}
In order to simplify the notation let us set $\Ff:=\Tt_{\PP^2}(-\log D)$. 
A line $L$ is a jumping line if $\mathrm{H}^0(\Ff(-1)\otimes \Oo_L)\neq 0$. Dualizing the resolution for $\Ff(1)$ and restricting it to $L$, we obtain 
\[
0\to \Ff(-1)\otimes \Oo_L  \to \Oo_{L}^{\oplus 3} \stackrel{\langle\nabla f\rangle_{|L}}\longrightarrow  \Oo_{L}(2)\to 0.
\]
The line $L$ is a jumping line if and only if the restricted map $(\nabla f)_{|L}$  is no longer an injection. This proves the first point.

According to \cite[Theorem 3.7]{Ma} the line $L$ is a singular point of $S_D$ if and only if 
$L$ is again a jumping line (actually the unique jumping line) for the stable bundle $\Ee \in \mathbf{M}_{\PP^2}(-1,2)$ defined by
the  exact sequence 
\[
0\to \Ee \to \Ff \to \Oo_L(-1)\to 0.
\]
Let us precise that the surjection above is just  the composition 
\[
\Ff \rightarrow \Ff\otimes \Oo_L  = \Oo_{L}(1)\oplus \Oo_{L}(-1) \rightarrow \Oo_{L}(-1)
\]
and that $\Ee$ is stable because 
$\mathrm{H}^0(\Ee)=0.$\\
 Choosing a global section of $\Tt_{\PP^2}(-\log D)\otimes \Oo_{\PP^2}(1)$, its zero locus is a complete intersection of two conics that are linear combinations of the derivatives of $f$, and $L$ is a line meeting this locus along two points. When this induces a section of $\Ee(1)$, its zero locus is two points again on $L$. This shows that the zero locus of the chosen  global section of $\Tt_{\PP^2}(-\log D)\otimes \Oo_{\PP^2}(1)$ is the complete intersection of 
 a double line supporting $L$ by another conic. Thus we have 
 \[ \ell^2\in \langle\nabla f\rangle.
 \]
The third assertion is proved in \cite[Proposition 5]{UY2}. 
\end{proof}

\begin{remark}\label{rmk-STbundle}
One can describe the jumping lines of $\Tt_{\PP^2}(-\log D)\cong \Omega_{\PP^2}^1(\log D)$ for a cubic curve $D$ of a Sebastiani-Thom type as before. Indeed, recall that any smooth cubic curve is projectively equivalent to a Hesse cubic of the form $x^3+y^3+z^3+3axyz=0$ for some $a\in \CC$ (see \cite{D}) and when $D$ is of Sebastiani-Thom type, the corresponding Hesse equation is the Fermat curve $x^3+y^3+z^3=0$. In the system of conics $\langle x^2, y^2, z^2\rangle$ spanned by the partial derivatives of $f$, there are exactly three double lines, and this implies that $S_D$ is the triangle defined by $\alpha_0\alpha_1\alpha_2=0$. Here the line $\alpha_0=0$ is dual to the point $[x:y:z]=[1:0:0]$ and similarly $\alpha_1, \alpha_2$ are defined. Indeed, each line in the triangle passes through $3$ points out of $\ell_1, \dots, \ell_9$. We get the same assertion for smooth cubic curves defined by $a_0x^3+a_1y^3+a_2z^3=0$ with each $a_i\in \CC^{\times}$. More explicitly, we have a resolution given by
\[
0 \to \Oo_{\PP^2}(-3) \stackrel{\mathrm{M}}{\longrightarrow} \Oo_{\PP^2}(-1)^{\oplus 3} \to \Tt_{\PP^2}(-\log D) \to 0,
\]
with $\mathrm{M}=
\begin{pmatrix}
x^2 & y^2 & z^2 \end{pmatrix}^t$. By direct computation, we have that
\[
 \left(\Tt_{\PP^2}(-\log D)\right)_{|L} \cong
\left\{
\begin{array}{cll}
    \Oo_L(-1) \oplus  \Oo_L(1) & \mbox{if} \:\: L \cap \left\{[1:0:0], [0:1:0], [0:0:1]\right\} \neq \emptyset;   \vspace{3mm}\\
    \Oo_L^{\oplus 2} & \mbox{for every other line} \:\: L. 
\end{array}
\right.
\]

In particular, the Torelli property does not hold for cubic curves of Sebastiani-Thom type.
\end{remark}

\bigskip

As a direct consequence of the previous description of the jumping lines of the logarithmic sheaf associated to a Sebastiani-Thom type cubic, we can observe the following situation: whenever the set of jumping lines $S_{D_1}$ and $S_{D_2}$ of the respective logarithmic sheaves of two Sebastiani-Thom type cubics $D_1$ and $D_2$ are different, then we have that $\Omega_{\PP^2}^1(\log  D_1)\not\cong \Omega_{\PP^2}^1(\log D_2)$. Therefore, it is enough to consider the Torelli problem for divisors which share the same triangle of jumping lines. Up to a projective transformation, we can suppose the triangle to be defined by $xyz=0$ and the inflection points of the cubic on the line defined by $z=0$ to be $p_i=[1:\eta^i : 0]$, for $i=0,1,2$ and $\eta$ a third root of the unity. In other words, one may consider the family of Sebastiani-Thom type cubics defined as
\[
\mathrm{ST}(3)_0 = \left\{ V(x^3-y^3+az^3) \:|\: a\in \CC^{\times} \right\}.
\]
Consider also the set of three points $W=\{[1:0:0], [0:1:0], [0:0:1]\}$, which do not belong to any curve in $\mathrm{ST}(3)_0$. In the next theorem we are going to see that the generalized logarithmic sheaf $\Omega_{\PP^2}^1(\log (D,Z))$ does not determine the cubic curve $D\in \mathrm{ST}(3)_0$, unless $Z\subset D$ already determines $D$ unambiguously. On the other hand, in Subsection $5.3$ we will see that the Torelli property holds for the strict transforms of the curves from $\mathrm{ST}(3)_0$ in the blown-up surface $\mathrm{Bl}_W \PP^2$.

\bigskip

\begin{theorem}\label{STP}
The Torelli property does not hold for pairs $(D,Z)$ of  cubic curves $D$ of Sebastiani-Thom type and points $Z$ on them, unless $Z$ determines uniquely $D$.\\
Moreover, for any cubic curve $D$ of Sebastiani-Thom type and any set $Z$ of three aligned inflection points of $D$, the pair $(D,Z)$ belongs to a one-dimensional family of curves that share the same generalized logarithmic sheaf. \end{theorem}

Let $D$ be a smooth cubic curve of Sebastiani-Thom type. As observed before, after a proper change of coordinates, we can suppose that $D$ belongs to the one-dimensional family $\mathrm{ST}(3)_0$, consisting of the cubic curves $D_a$ of Sebastiani-Thom type, defined as
\[
D_a=V(x^3-y^3+az^3)
\]
with $a\in \CC^{\times}$.  All of $D_a$ share the same three inflection points $p_i=[1:\eta^i : 0]$, for $i=0,1,2$ and $\eta$ a third root of unity, on the line $z=0$. Furthermore, the one-dimensional family we have just considered describes all of cubics of Sebastiani-Thom type that share these inflections points with the fixed curve $D$.\\
By direct computation and as observed in Remark \ref{rmk-STbundle} (or also explained in \cite[Remark 8]{UY2}), all of these curves have the same associated logarithmic vector bundle.\\ 

Now we fix points $Z\subset D$ and investigate the Torelli property for all possible pair $(D_a, Z)$. Notice that if $Z$ contains a non-inflection point, there exists a unique curve $D_a \in \mathrm{ST}(3)_0$ passing through $Z$ and then the Torelli property obviously holds. Notice also that, analogously, the curve is uniquely determined when $Z$ is formed by inflection points not contained in the line $z=0$. Thus we may assume $Z\subset \{p_0, p_1, p_2\}$, i.e., a subset of the three inflection points on the line $z=0$.

\begin{remark}
We are aware that, if we prove that the Torelli property does not hold for the pair $(D,\{p_0,p_1,p_2\})$, i.e., considering all three inflection points, then, due to Proposition \ref{prop-Torelliplusone}, the Torelli property does not hold automatically for $(D,Z)$ for any $Z\subset \{p_0,p_1,p_2\}$. Nevertheless, we will see that the proposed method of the proof is constructive, in the sense that we start by fixing one inflection point and afterwards we add a second and a third one, relating the given generalized logarithmic sheaves from two consecutive steps.
\end{remark}

\vspace{.3cm}

\noindent\textbf{(A) \quad Case $|Z|=1$}

Let us first consider all possible pairs $(D_a,Z)$, being $Z=\{p_0\}$ (recall $p_0=[1:1:0])$, for any cubic curve $D_a$ in the family described above.\\ 
Then, in this case, Theorem \ref{STP} follows from the following result, that shows how the defining matrix of the associated generalized logarithmic sheaf is independent of the value of $a$.

\begin{proposition}\label{prop555}
The minimal free resolution of $\:\:\Omega_{\PP^2}^1(\log (D_a, \{p_0\}))$, for any $a\in \CC^{\times}$, is of the form
\[
0\to \Oo_{\PP^2}(-3)^{\oplus 2} \stackrel{\mathrm{\tilde{N}}}{\longrightarrow} \Oo_{\PP^2}(-2)^{\oplus 2}\oplus  \Oo_{\PP^2}(-1)^{\oplus 2} \to \Omega_{\PP^2}^1(\log (D_a,\{p_0\})) \to 0,
\]
with
\[
\mathrm{\tilde{N}}=
\begin{pmatrix}  
y-x & z & 0 & 0\\ 
0 & x+y & y^2  & z^2 \\ 
\end{pmatrix}^t.
\]
\end{proposition}

\begin{proof}

We start by applying the Horseshoe Lemma (see \cite[Lemma 2.2.8]{weibel}), which can be applied as in Remark \ref{keyrem}, to the short exact sequence
\begin{equation}\label{a1}
0\to\Omega_{\PP^2}^1\to \Omega_{\PP^2}^1(\log (D_a,\{p_0\}))  \to\Ii_{\{p_0\}, D_a}\to 0
\end{equation}
in order to find a resolution for $\Omega_{\PP^2}^1(\log (D_a,\{p_0\}))$. From the short exact sequence
\[
0\to\Ii_{D_a, \PP^2} \to\Ii_{\{p_0\}, \PP^2} \to \Ii_{\{p_0\}, D_a}\to 0,
\]
 that relates the two ideal sheaves $\Ii_{D_a, \PP^2}$ and $\Ii_{\{p_0\}, \PP^2}$, respectively of the cubic curve and of one point in $\PP^2$, with $\Ii_{\{p_0\}, D_a}$, the ideal sheaf of the point in the curve, it is possible to construct the following resolution of $\Ii_{\{p_0\}, D_a}$, 
\[
0\to \Oo_{\PP^2}(-3)\oplus\Oo_{\PP^2}(-2)\stackrel{\mathrm{M}_0}{\longrightarrow}\Oo_{\PP^2}(-1)^{\oplus 2} \to\Ii_{\{p_0\} , D_a}\to 0,
\]
where 
\[
\mathrm{M}_0=\begin{pmatrix}  x^2+xy+y^2 & az^2 \\ -z  & x-y  \end{pmatrix}^t.
\]
From this, applying the Horseshoe Lemma at the sequence (\ref{a1}), we get the resolution 
\begin{equation}\label{epp1}
0\to \Oo_{\PP^2}(-3)^{\oplus 2} \oplus\Oo_{\PP^2}(-2)\stackrel{\mathrm{N}}{\longrightarrow }\Oo_{\PP^2}(-2)^{\oplus 3} \oplus\Oo_{\PP^2}(-1)^{\oplus 2} \to \Omega_{\PP^2}^1(\log (D_a,\{p_0\})) \to 0
\end{equation}
with
\[
\mathrm{N}=\begin{pmatrix}  x & y & z & 0 & 0\\ l_1 & l_2 &l_3 & x^2+xy+y^2 & az^2 \\ m_1& m_2 & m_3 & -z & x-y  \end{pmatrix}^t
\]
with $m_i$ forms of degree zero and $l_i$ forms of degree one, for $i=1,2,3$.\\ 
 Notice that, being $\Omega_{\PP^2}^1(\log (D_a,\{p_0\}))$ not locally free along $p_0$ (see Proposition \ref{prop234}), the evaluation of $\mathrm{N}$ at the point $p_0=[1:1:0]$ cannot have maximal rank. This directly implies that $m:=m_1=m_2$ and $m_3=0$.\\ 
Our next goal is to prove that the coefficients $m_1$ and $m_2$ are different from zero, which implies that the resolution (\ref{epp1}) can be ``simplified''. Suppose it cannot be, i.e., let us suppose that $m_1 = m_2 = m_3 = 0$. Under such assumption, we have the following commutative diagram

\begin{equation}\label{diag-simplify}
\begin{tikzcd}[
 column sep=small,row sep=normal,
  ar symbol/.style = {draw=none,"\textstyle#1" description,sloped},
  isomorphic/.style = {ar symbol={\cong}},
  ]
& 0 \ar[d] & 0 \ar[d] & 0\ar[d] \\
0 \ar[r] & \Oo_{\PP^2}(-2) \ar[d] \ar[r,"\mathrm{A}"] &  \Oo_{\PP^2}(-1)^{\oplus 2} \ar[r] \ar[d] & \mathcal{I}_{\{p_0\}, \PP^2} \ar[r] \ar[d] & 0\\
0 \ar[r] & \Oo_{\PP^2}(-3)^{\oplus 2} \oplus \Oo_{\PP^2}(-2) \ar[d] \ar[r] & \Oo_{\PP^2}(-2)^{\oplus 3} \oplus \Oo_{\PP^2}(-1)^{\oplus 2} \ar[r] \ar[d] &  \Omega_{\PP^2}^1(\log (D_a,\{p_0\})) \ar[r] \ar[d] & 0\\
0\ar[r] & \Oo_{\PP^2}(-3)^{\oplus 2} \ar[r, "\mathrm{B}"] \ar[d] & \Oo_{\PP^2}(-2)^{\oplus 3} \ar[r] \ar[d] & \mathcal{G} \ar[r] \ar[d] & 0\\
& 0 & 0 & 0
\end{tikzcd}
\end{equation}
for $\mathrm{A}=\begin{pmatrix} -z & x-y\end{pmatrix}$ and $\mathrm{B}=\begin{pmatrix} x&y&z\\ l_1 & l_2 & l_3 \end{pmatrix}$. 
Furthermore, considering the rank 1 coherent sheaf $\Gg$ obtained in the previous diagram, take its associated canonical exact sequence defined by
\begin{equation}\label{diag-resolution}
\begin{tikzcd}[
 column sep=normal,row sep=small,
  ar symbol/.style = {draw=none,"\textstyle#1" description,sloped},
  isomorphic/.style = {ar symbol={\cong}},
  ]
0\ar[r] &\Tt_{\Gg} \ar[r]& \Gg \ar[rr] \ar[rd] && \Gg^{\vee\vee} \ar[r] & \Qq \ar[r] & 0\\
&&&\Ii_{\Delta, \PP^2}(-\alpha) \ar[ru]\ar[rd]\\
&& 0\ar[ru] &&0
\end{tikzcd}
\end{equation}
being $\Tt_{\Gg}$ the maximal torsion subsheaf of $\Gg$ and $\Gg^{\vee\vee}$ the double dual of $\Gg$ which is known to be reflexive and, therefore, a locally free sheaf in this case (see \cite[Section II.1]{OSS}). This implies that the splitting of the exact sequence gives us a sheaf of ideals $\Ii_{\Delta, \PP^2}(-\alpha)$, with $\Delta$ a $0$-dimensional scheme in the projective plane and $\alpha \geq 0$.\\ 
 Restricting the right vertical sequence in Diagram (\ref{diag-simplify}) and Sequence (\ref{diag-resolution}) to the generic line $L$ (not passing through $\{p_0\}$) in the projective plane, we obtain the following composition of surjections
\[
\left(\Omega_{\PP^2}^1(\log (D_a,\{p_0\}))\right)_{|L} \simeq \left(\Omega_{\PP^2}^1(\log D_a) \right)_{|L} \simeq \Oo_L^
{\oplus 2} \twoheadrightarrow \mathcal{G}_{|L} \twoheadrightarrow \left(\mathcal{I}_{\Delta, \PP^2}(-\alpha)\right)_{|L} \simeq \Oo_L(-\alpha).
\]
This implies that, necessarily, $\alpha=0$ and therefore $\mathcal{T}_\mathcal{G}$ is supported on a 0-dimensional scheme. From the fact that $\mathrm{H}^0\left(\mathcal{G}\right) = 0$, computed directly from the bottom row of Diagram (\ref{diag-simplify}), we get $\mathcal{T}_\mathcal{G} = 0$ and hence $\mathcal{G} \simeq \mathcal{I}_{\Delta, \PP^2}$. From such isomorphism, dualizing the right vertical sequence in  Diagram (\ref{diag-simplify}), we would get
\[
0 \rightarrow \Oo_{\PP^2} \rightarrow \left(\Omega_{\PP^2}^1(\log (D_a,\{p_0\}))\right)^{\lor} \simeq \Tt_{\PP^2}(-\log D_a),
\]
which leads to contradiction, because the logarithmic tangent bundle associated to a Sebastiani-Thom cubic has no global section.\\

As a consequence, $m=m_1=m_2 \neq 0$ and we can simplify the resolution in (\ref{epp1}) to get
\[
0\to \Oo_{\PP^2}(-3)^{\oplus 2} \stackrel{\mathrm{\widetilde{N}}}{\longrightarrow }\Oo_{\PP^2}(-2)^{\oplus 2} \oplus\Oo_{\PP^2}(-1)^{\oplus 2} \to \Omega_{\PP^2}^1(\log (D_a,\{p_0\}))\to 0
\]
with
\[
\mathrm{\widetilde{N}}=
\begin{pmatrix}  
y-x & z & 0 & 0\\ 
l_2 - l_1 & l_3 & x^2+xy+y^2 + \frac{l_1 z}{m} - \frac{l_3 x}{m}  & az^2 +
\frac{l_1 y}{m} - \frac{l_2 x}{m} \\ 
\end{pmatrix}^t.
\]
This matrix, and in particular its non zero entries in the first row, implies an injective map 
$$
0\rightarrow \Ii_{\{p_0\}, \PP^2}(-1) \rightarrow \Omega_{\PP^2}^1(\log (D_a,\{p_0\})),
$$ 
for which we can consider the associated injective map of graded modules $$
0\rightarrow \mathrm{I}:=\oplus\mathrm{H}^0_*(\Ii_{\{p_0\}, \PP^2}(-1))\stackrel{\widehat{\varphi}}{\longrightarrow}\mathrm{F}:=\mathrm{H}^0_*(\Omega_{\PP^2}^1(\log (D_a,\{p_0\}))$$ and denote by $\mathrm{K}$ its cokernel module. From the description of the matrix $\widetilde{\mathrm{N}}$, it is possible to apply the mapping cone procedure to obtain a graded free resolution of the module $\mathrm{K}$, that is  of the form
\[
0\to S(-3)\stackrel{(q_1 \, q_2)}{\xrightarrow{\hspace*{1.2cm}}}S(-1)^{\oplus 2}\to \mathrm{K}\to 0
\]
where $q_1=x^2+xy+y^2 + \frac{l_1 z}{m} - \frac{l_3 x}{m} $ and $q_2=az^2 +
\frac{l_1 y}{m} - \frac{l_2 x}{m}$. Therefore, $\mathrm{K}$ is the graded ideal $I_Q$ of the complete intersection $Q$ defined by the two conics $q_1$ and $q_2$, twisted by one. Sheafifying the considered short exact sequence of modules, we have that the sheaf $\Omega_{\PP^2}^1(\log (D_a,\{p_0\})$ can also be described as an extension
\begin{equation}\label{werrw1}
0 \to \Ii_{\{p_0\}, \PP^2}(-1) \stackrel{\varphi}{\longrightarrow} \Omega_{\PP^2}^1(\log (D_a,\{p_0\}) \to \Ii_{Q, \PP^2}(1) \to 0, 
\end{equation}
where $Q$ is the 0-dimensional scheme of length $4$ obtained before. 
Moreover, because of the resolution of $\Omega_{\PP^2}^1(\log (D_a,\{p_0\})$, the two forms $q_1$ and $q_2$ cannot have a linear common factor. \\
Dualizing Sequence (\ref{werrw1}), we obtain a long exact sequence
\begin{equation}\label{aaa}
\begin{array}{l}
    0 \to \Oo_{\PP^2}(1) \to \Omega_{\PP^2}^1(\log D_a) \to \Oo_{\PP^2}(1) \to \mathcal{E}xt_{\PP^2}^1(\Ii_{Q, \PP^2}(1),\Oo_{\PP^2})\\
     \to \mathcal{E}xt_{\PP^2}^1(\Omega_{\PP^2}^1(\log (D_a,\{p_0\}),\Oo_{\PP^2})\stackrel{\psi}{\longrightarrow} \mathcal{E}xt_{\PP^2}^1(\Ii_{\{p_0\}, \PP^2}(-1),\Oo_{\PP^2})\to \mathcal{E}xt_{\PP^2}^2(\Ii_{Q, \PP^2}(1),\Oo_{\PP^2})\cong 0.
\end{array}
\end{equation}
Furthermore, being $p_0$ and $Q$ two 0-dimensional schemes in the projective plane, we have that 
\[
\mathcal{E}xt_{\PP^2}^1(\Ii_{Q, \PP^2}(1),\Oo_{\PP^2})\cong \mathcal{E}xt_{\PP^2}^2(\Oo_Q,\Oo_{\PP^2})\cong\Oo_Q.
\]
and 
$$
\mathcal{E}xt_{\PP^2}^1(\Ii_{\{p_0\}, \PP^2}(-1),\Oo_{\PP^2})\cong
\mathcal{E}xt_{\PP^2}^2(\Oo_{p_0},\Oo_{\PP^2})\cong\Oo_{p_0}.
$$
Consider now the short exact sequence (given by Proposition (\ref{prop234})
\begin{equation}\label{ppp222}
0 \to \Omega_{\PP^2}^1(\log (D_a,\{p_0\}) \to \Omega_{\PP^2}^1(\log (D_a) \simeq \Tt_{\PP^2}(-\log D_a) \to \Oo_{p_0} \to 0.
\end{equation}
and, dualizing it, we obtain that $\mathcal{E}xt_{\PP^2}^1(\Ff,\Oo_{\PP^2})\cong\Oo_{p_0}$ as well.\\ 
Putting everything together, we get that the map $\psi$ in (\ref{aaa}) is an isomorphism and the splitting of the sequence (\ref{aaa}) gives us
\begin{equation}\label{aa2}
0 \to \Oo_{\PP^2}(-1) \to \Omega_{\PP^2}^1(\log D_a) \to \Ii_{Q, \PP^2}(1) \to 0.
\end{equation}
Recall that the resolution of $\Omega_{\PP^2}^1(\log D_a)$ was of the form
\[
0\to \Oo_{\PP^2}(-3)\stackrel{\mathrm{M}_1}{\longrightarrow}\Oo_{\PP^2}(-1)^{\oplus 3}\to \Omega_{\PP^2}^1(\log D_a)\to 0.
\]
with $\mathrm{M}_1=\begin{pmatrix} x^2 & y^2 & z^2\end{pmatrix}^t$. From this fact and the exact sequence (\ref{aa2}), it is possible to apply the mapping cone to find a resolution of $\Ii_{Q, \PP^2}(1)$, which implies that $Q$ is a zero-dimensional scheme of length $4$
whose ideal $I_Q$ is defined by two of the entries in $\mathrm{M}_1$. Since this ideal should also be defined by the conics in the matrix $\widetilde{\mathrm{N}}$ we can conclude, comparing coefficients, that the only possibility is
$I_Q=(y^2,z^2)$ and that the matrix $\widetilde{\mathrm{N}}$ can be taken of the form
\[
\mathrm{\widetilde{N}}=
\begin{pmatrix}  
y-x & z & 0 & 0\\ 
0 & x+y & y^2  & z^2 \\ 
\end{pmatrix}^t. \qedhere
\]
\end{proof}

\vspace{.3cm}

\noindent\textbf{(B) \quad Case $|Z|=2$}

Without loss of generality we may assume that $Z=\{p_0, p_1\}$ and consider the short exact sequence (see the proof of Proposition \ref{prop-Torelliplusone})
\[
0\to \Omega_{\PP^2}^1(\log (D_a, \{p_0, p_1\})) \to \Omega_{\PP^2}^1(\log (D_a, \{p_0\})) \stackrel{g}{\longrightarrow} \Oo_{p_1} \to 0.
\]
Consider the injective map $\varphi$ in Sequence (\ref{werrw1}) to construct the following composition 
$$
\rho: \Ii_{\{p_0\}, \PP^2}(-1) \stackrel{\varphi}{\rightarrow} \Omega_{\PP^2}^1(\log (D_a, \{p_0\})) \stackrel{g}{\longrightarrow} \Oo_{p_1}.
$$
We will now discuss the two subcases given by $\rho$ being or not zero.

\vspace{.3cm}

\noindent\textbf{(Case B-1)} 

Assume first that $\rho$ is not zero. Then it is surjective and its kernel is isomorphic to $\Ii_{\{p_0,p_1\}, \PP^2}(-1)$. This produces, because of the vanishing $\mathrm{Ext}_{\PP^2}^1(\Oo_{p_1}, \Ii_{Q, \PP^2}(1))=0$, a commutative diagram\\

\[
\begin{tikzcd}[
 column sep=small,row sep=small,
  ar symbol/.style = {draw=none,"\textstyle#1" description,sloped},
  isomorphic/.style = {ar symbol={\cong}},
  ]
 &&0 \ar[d]& 0\ar[d]  \\
0\ar[r] & \Ii_{\{p_0,p_1\}, \PP^2}(-1) \ar[r]\ar[d, isomorphic]&  \Omega_{\PP^2}^1(\log (D_a, \{p_0, p_1\})) \ar[r] \ar[d] &  \Ii_{Q, \PP^2}(1) \ar[r]\ar[d] & 0\\
0\ar[r] &\Ii_{\{p_0,p_1\}, \PP^2}(-1)  \ar[r] & \Omega_{\PP^2}^1(\log (D_a, \{p_0\})) \ar[r] \ar[d] & \Ii_{Q, \PP^2}(1)\oplus \Oo_{p_1} \ar[r] \ar[d]& 0 \\
&&  \Oo_{p_1} \ar[d] \ar[r, isomorphic] & \Oo_{p_1}\ar[d] \\
& &0 &0,
\end{tikzcd}\]
being $Q$ the same 0-dimensional scheme as in Sequence (\ref{werrw1}).\\ 
Apply now the functor $\mathrm{Hom}_{\PP^2}(-, \Ii_{\{p_0,p_1\}, \PP^2}(-1))$ to the right vertical sequence to have a surjection
\begin{equation}\label{def-Psiextension}
\Psi : \mathrm{Ext}_{\PP^2}^1(\Ii_{Q, \PP^2}(1)\oplus \Oo_{p_1},\Ii_{\{p_0,p_1\}, \PP^2}(-1)) \to  \mathrm{Ext}_{\PP^2}^1(\Ii_{Q, \PP^2}(1),\Ii_{\{p_0,p_1\}, \PP^2}(-1))
\end{equation}
whose kernel $\mathrm{Ext}_{\PP^2}^1(\Oo_{p_1},\Ii_{\{p_0,p_1\}, \PP^2}(-1))$ is one-dimensional.\\ 
To ensure that the Torelli property does not hold also for this case, it is sufficient to prove the following result.
\begin{lemma}\label{lem-uniquext}
There exists, for any $a\in \CC^{\times}$ fixed, a unique presentation of $\:\Omega_{\PP^2}^1(\log(D_a,\{p_0\}))$ as an extension element, up to a scalar, in $\Ext_{\PP^2}^1(\Ii_{Q, \PP^2}(1)\oplus \Oo_{p_1},\Ii_{\{p_0,p_1\}, \PP^2}(-1))$. 
\end{lemma}
Indeed, fix two different cubics $D_{a_1}$ and $D_{a_2}$, sharing the same inflection points $p_0$ and $p_1$. Their associated logarithmic sheaves $\Omega_{\PP^2}^1(\log (D_{a_1}, \{p_0, p_1\}))$ and $\Omega_{\PP^2}^1(\log (D_{a_2}, \{p_0, p_1\}))$ are given respectively (seen as elements of the extension groups) as the image by $\Psi$ of the generalized logarithmic sheaves $\Omega_{\PP^2}^1(\log (D_{a_1}, \{p_0\}))$ and $\Omega_{\PP^2}^1(\log (D_{a_2}, \{p_0\}))$. We have, by Case (A), that
$$\Omega_{\PP^2}^1(\log (D_{a_1}, \{p_0\})) \simeq \Omega_{\PP^2}^1(\log (D_{a_2}, \{p_0\}))$$ and, being $\Psi$ surjective, we get that
$$\Omega_{\PP^2}^1(\log (D_{a_1}, \{p_0,p_1\})) \simeq \Omega_{\PP^2}^1(\log (D_{a_2}, \{p_0,p_1\}))$$
as well.\\

\begin{proof}[Proof of Lemma \ref{lem-uniquext}]
Recall that there exists a unique extension in $\Ext_{\PP^2}^1(\Ii_{Q, \PP^2}(1),\Ii_{\{p_0\}, \PP^2}(-1))$, up to a scalar, that corresponds to the logarithmic sheaf $\Omega_{\PP^2}^1(\log(D_a,\{p_0\}))$. This follows because, by Proposition \ref{prop555}, the sheaf is uniquely determined by the matrix
\[
\widetilde{\mathrm{N}}=\begin{pmatrix}  x-y & z & 0 & 0\\ 0 & x+y &y^2 & z^2 \end{pmatrix}^t.
\] 
Let us start considering the diagram 
\[
\begin{tikzcd}[
 column sep=normal,row sep=normal,
  ar symbol/.style = {draw=none,"\textstyle#1" description,sloped},
  isomorphic/.style = {ar symbol={\cong}},
  ]
0\ar[r] & \Ii_{\{p_0,p_1\}, \PP^2}(-1) \ar[r,"\widetilde{\phi}_1"] & \Omega_{\PP^2}^1(\log(D_a,\{p_0\})) \ar[r,"\widetilde{\phi}_2"] \ar[d,"\mathrm{Id}"] & \Ii_{Q, \PP^2}(1)\oplus\Oo_{p_1} \ar[r]\ar[d] & 0\\
 &  &  \Omega_{\PP^2}^1(\log(D_a,\{p_0\})) \ar[r,"\phi_2"] & \Ii_{Q, \PP^2}(1) \ar[r] & 0.
\end{tikzcd}
\]
Since $\Hom_{\PP^2}(\Ii_{\{p_0,p_1\}, \PP^2}(-1),\Ii_{Q, \PP^2}(1))=0$, $\Ext^1(\Oo_{p_1},\Ii_{\{p_0,p_1\}, \PP^2}(-1))\cong \CC$ and that $\Omega_{\PP^2}^1(\log(D_a,\{p_0\}))$ is torsion-free, we can complete the previous diagram to have the following one
\[
\begin{tikzcd}[
 column sep=normal,row sep=normal,
  ar symbol/.style = {draw=none,"\textstyle#1" description,sloped},
  isomorphic/.style = {ar symbol={\cong}},
  ]
0\ar[r] & \Ii_{\{p_0,p_1\}, \PP^2}(-1) \ar[r,"\widetilde{\phi}_1"] \ar[d,"i"] & \Omega_{\PP^2}^1(\log(D_a,\{p_0\})) \ar[r,"\widetilde{\phi}_2"] \ar[d,"\mathrm{Id}"] & \Ii_{Q, \PP^2}(1)\oplus\Oo_{p_1} \ar[r]\ar[d] & 0\\
0\ar[r] & \Ii_{\{p_0\}, \PP^2}(-1)\ar[r,"\phi_1"] \ar[d] &  \Omega_{\PP^2}^1(\log(D_a,\{p_0\})) \ar[r,"\phi_2"] & \Ii_{Q, \PP^2}(1) \ar[r] & 0 \\
 & \Oo_{p_1}. & &  & 
\end{tikzcd}
\]
Begin $\varphi_1$ and $\varphi_2$ uniquely determined by the matrix $\widetilde{\mathrm{N}}$, we obtain from the diagram that $\widetilde{\phi}_1=\phi_1\circ i$ is also uniquely determined and, writing $\widetilde{\phi}_2=(\widetilde{\phi}_{2,1},\widetilde{\phi}_{2,2})$, $\tilde{\phi_{2,1}}=\phi_2$ is also uniquely determined.

To conclude, let $e\in\Ext_{\PP^2}^1(\Ii_{Q, \PP^2}(1),\Ii_{\{p_0,p_1\}, \PP^2}(-1))$ correspond to the extension that defines $\Omega_{\PP^2}^1(\log(D_a,\{p_0,p_1\}))$. Clearly, taking the map $\Psi$ defined in (\ref{def-Psiextension}), $\Psi^{-1}(e)$ is not unique, but any element in this preimage corresponds to an extension $\Ff$ of the form
\[
o\to\Ii_{\{p_0,p_1\}, \PP^2}(-1)\to \Ff \stackrel{\widetilde{\zeta}}{\longrightarrow} \Ii_{Q, \PP^2}(1)\oplus\Oo_{p_1}\to 0
\]
with $\widetilde{\zeta}=(\widetilde{\zeta}_1,\widetilde{\zeta}_2)$. Notice that for $\widetilde{\zeta}_1=\varphi_2$ is fixed and $\widetilde{\zeta}_2$ is determined by the extension
\[
0\to\Ii_{\{p_0,p_1\},\PP^2}(-1)\to \Gg \to \Oo_{p_1}\to 0
\]
in $\Ext_{\PP^2}^1(\Oo_{p_1},\Ii_{\{p_0,p_1\}, \PP^2}(-1))\cong \CC$.\\ 
We can therefore conclude that the only logarithmic sheaf in $\Psi^{-1}(e)$ is $\Omega_{\PP^2}^1(\log(D_a,\{p_0\}))$.
\end{proof}

\vspace{.3cm}

\noindent\textbf{(Case B-2)} 
Let us assume that $\rho$ is zero. Then the induced composition $\widetilde{\rho}: \Oo_{\PP^2}(-1) \rightarrow \Omega_{\PP^2}^1(\log D_a) \rightarrow \Oo_{p_1}$ is also zero and we can obtain the following commutative diagram
\[
\begin{tikzcd}[
 column sep=small,row sep=small,
  ar symbol/.style = {draw=none,"\textstyle#1" description,sloped},
  isomorphic/.style = {ar symbol={\cong}},
  ]
 &&0 \ar[d]& 0\ar[d]  \\
0\ar[r] & \Oo_{\PP^2}(-1) \ar[r]\ar[d, isomorphic]&  \Omega_{\PP^2}^1(\log (D_a, \{p_1\})) \ar[r] \ar[d] &  \Ii_{Q \cup \{p_1\}, \PP^2}(1) \ar[r]\ar[d] & 0\\
0\ar[r] &\Oo_{\PP^2}(-1)  \ar[r] & \Omega_{\PP^2}^1(\log D_a) \ar[r] \ar[d] & \Ii_{Q, \PP^2}(1) \ar[r] \ar[d]& 0 \\
&&  \Oo_{p_1} \ar[d] \ar[r, isomorphic] & \Oo_{p_1}\ar[d] \\
& &0 &0.
\end{tikzcd}\]
 We will conclude this case by showing, in the subsequent result, that an exact sequence as upper horizontal one in the diagram is impossible. Observe that the following lemma also holds if we  replace $p_0$ by $p_1$. 

\begin{lemma}\label{erttt1}
With the notation in the above, one cannot have an exact sequence
\[
0\to \Oo_{\PP^2}(-1) \to \Omega_{\PP^2}^1(\log (D_a, \{p_0\})) \to \Ii_{Q \cup \{p_0\}, \PP^2}(1) \to 0.
\]
\end{lemma}

\begin{proof}
Applying the functor $\mathrm{Hom}_{\PP^2}(\Oo_{\PP^2}(-1), -)$ to the resolution, proven in Proposition \ref{prop555},
\[
0\to \Oo_{\PP^2}(-3)^{\oplus 2} \stackrel{\mathrm{N}}{\longrightarrow} \Oo_{\PP^2}(-2)^{\oplus 2} \oplus \Oo_{\PP^2}(-1)^{\oplus 2}\to \Omega_{\PP^2}^1(\log (D_a, \{p_0\})) \to 0,
\]
we can assure that the injection $$\Oo_{\PP^2}(-1) \hookrightarrow \Omega_{\PP^2}^1(\log (D_a, \{p_0\}))$$ lifts to an injection $$\Oo_{\PP^2}(-1) \hookrightarrow \Oo_{\PP^2}(-2)^{\oplus 2} \oplus \Oo_{\PP^2}(-1)^{\oplus 2}.$$ Therefore, we get the following short exact sequence
\[
0\to \Oo_{\PP^2}(-3)^{\oplus 2} \stackrel{\mathrm{M}_2}{\longrightarrow} \Oo_{\PP^2}(-2)^{\oplus 2} \oplus \Oo_{\PP^2}(-1) \to \Ii_{Q \cup \{p_0\}, \PP^2}(1) \to 0.
\]
Here, the matrix $\mathrm{M}_2$ comes from the matrix
\[
\widetilde{\mathrm{N}}=\begin{pmatrix} y-x & z & 0& 0\\ 0&x+y&y^2 & z^2\end{pmatrix}^t,
\]
deleting one of the two last columns after eventual combinations of its columns. This means that $\mathrm{M}_2$ has the following form:
\[
\mathrm{M}_2=\begin{pmatrix} y-x & z & * \\0&x+y & *\end{pmatrix}^t,
\]
because, as mentioned before, we have to cancel a combination
\[
\ell_1 \begin{pmatrix} y-x \\ 0\end{pmatrix}+\ell_2\begin{pmatrix} z\\x+y\end{pmatrix}+\alpha_1\begin{pmatrix} 0\\y^2\end{pmatrix}+\alpha_2 \begin{pmatrix} 0\\z^2\end{pmatrix}
\]
with $\ell_i$ a linear form and $\alpha_j$ a constant, corresponding to the deleted $\Oo_{\PP^2}(-1)$-factor. 

Consider now the two ideals defining $p_0$ and $Q$, respectively, 
\[
\mathrm{I}=(y-x, z), \quad \mathrm{J}=(y^2, z^2),
\]
and so the saturation of their product is as follows 
\[
\mathrm{Sat}(\mathrm{I}\cdot \mathrm{J})=\mathrm{Sat}(\mathrm{I}\cap \mathrm{J})=(z^2, y^2z, xy^2-y^3).
\]
Then, using Macaulay 2, one can obtain the syzygy matrix $\mathrm{N_0}=\begin{pmatrix} y-x & z& 0\\ z&0&-y^2 \end{pmatrix}^t$ fitting into the following diagram
\[
\begin{tikzcd}[
 column sep=small,row sep=small,
  ar symbol/.style = {draw=none,"\textstyle#1" description,sloped},
  isomorphic/.style = {ar symbol={\cong}},
  ]
0\ar[r] & \Oo_{\PP^2}(-3) \ar[r, "\mathrm{N_0}"] & [1.5em]\Oo_{\PP^2}(-2)^{\oplus 2} \oplus \Oo_{\PP^2}(-1) \ar[rr, "\mathrm{Sat}(\mathrm{I}\cdot \mathrm{J})"] \ar[rd] && \Oo_{\PP^2}(1)\\
&&&\Ii_{Q\cup \{p_0\}, \PP^2}(1) \ar[ru]
\end{tikzcd}\]
Combining the two obtained resolutions for $\Ii_{Q \cup \{p_0\}, \PP^2}(1)$, we get the following diagram
\[
\begin{tikzcd}[
 column sep=small,row sep=small,
  ar symbol/.style = {draw=none,"\textstyle#1" description,sloped},
  isomorphic/.style = {ar symbol={\cong}},
  ]
0\ar[r] & \Oo_{\PP^2}(-3) \ar[d, "\psi_2"] \ar[r, "\mathrm{N}"] & [1.5em]\Oo_{\PP^2}(-2)^{\oplus 2} \oplus \Oo_{\PP^2}(-1) \ar[d, "\psi_1"] \ar[r] &\Ii_{Q \cup \{p_0\}, \PP^2}(1) \ar[r]\ar[d, isomorphic] & 0\\[1.5ex]
0\ar[r] & \Oo_{\PP^2}(-3) \ar[r, "\mathrm{N_0}"] & [1.5em]\Oo_{\PP^2}(-2)^{\oplus 2} \oplus \Oo_{\PP^2}(-1) \ar[r] &\Ii_{Q \cup \{p_0\}, \PP^2}(1) \ar[r] &0
\end{tikzcd}\]
If $\psi_2$ is not an isomorphism, then we have $\mathrm{Coker}(\psi_2) \cong \mathrm{Coker}(\psi_1) \cong \Oo_{\PP^2}(-3)$, which is absurd. Thus each $\psi_i$, for $i=1,2$, is an isomorphism and hence represented by the following two invertible matrices 
\[
\mathrm{C}=\begin{pmatrix} \mu_1 & \mu_2 & 0\\ \mu_3 & \mu_4 & 0 \\ \ell_1 & \ell_2 & \alpha\end{pmatrix}, \quad \mathrm{D}=\begin{pmatrix} \lambda_1 & \lambda_2 \\ \lambda_3 & \lambda_4 \end{pmatrix}. 
\]
By direct computation, the commutativity condition $\mathrm{C}\mathrm{N}=\mathrm{N_0}\mathrm{D}$ gives $\lambda_2=\lambda_4=0$, contradicting the invertibility of $\mathrm{D}$. 
\end{proof}

\vspace{.3cm}

\noindent\textbf{(C) \quad Case $|Z|=3$}
From the exact sequence
\begin{equation}
0\to \Omega_{\PP^2}^1(\log (D_a, \{p_0,p_1,p_2\})) \to \Omega_{\PP^2}^1(\log (D_a, \{p_0\})) \to \Oo_{p_1}\oplus \Oo_{p_2} \to 0,
\end{equation}
we can consider again a composition $\rho : \Ii_{\{p_0\}, \PP^2}(-1) \rightarrow \Omega_{\PP^2}^1(\log (D_a, \{p_0\})) \rightarrow \Ii_{Q, \PP^2}(1)$ as in Case \textbf{(B)}. Similarly as in \textbf{(Case B-2)} if $\rho$ is not surjective we get a contradiction, therefore we may assume that  $\rho$ is surjective. This implies the following commutative diagram
\[
\begin{tikzcd}[
 column sep=small,row sep=small,
  ar symbol/.style = {draw=none,"\textstyle#1" description,sloped},
  isomorphic/.style = {ar symbol={\cong}},
  ]
 &&0 \ar[d]& 0\ar[d]  \\
0\ar[r] & \Ii_{\{p_0,p_1,p_2\}, \PP^2}(-1) \ar[r]\ar[d, isomorphic]&  \Omega_{\PP^2}^1(\log (D_a, \{p_0, p_1, p_2\})) \ar[r] \ar[d] &  \Ii_{Q, \PP^2}(1) \ar[r]\ar[d] & 0\\
0\ar[r] &\Ii_{\{p_0,p_1,p_2\}, \PP^2}(-1)  \ar[r] & \Omega_{\PP^2}^1(\log (D_a, \{p_0\})) \ar[r] \ar[d] & \Ii_{Q, \PP^2}(1)\oplus  \Oo_{p_1}\oplus \Oo_{p_2} \ar[r] \ar[d]& 0 \\
&&  \Oo_{p_1}\oplus \Oo_{p_2} \ar[d] \ar[r, isomorphic] & \Oo_{p_1}\oplus \Oo_{p_2}\ar[d] \\
& &0 &0. 
\end{tikzcd}\]
Applying the same technique as in \textbf{(Case B-1)}, one can show that there exists a unique extension in $\mathrm{Ext}_{\PP^2}^1(\Ii_{Q, \PP^2}(1)\oplus \Oo_{p_1}\oplus \Oo_{p_2},  \Ii_{\{p_0,p_1,p_2\}, \PP^2}(-1))$, up to an extension inside $\mathrm{Ext}_{\PP^2}^1(\Oo_{p_1}\oplus \Oo_{p_2}, \Ii_{\{p_0,p_1,p_2\}, \PP^2}(-1))$, that represents the generalized logarithmic sheaf $\Omega_{\PP^2}^1(\log (D_a, \{p_0\}))$. Finally, by the same argument as in \textbf{(Case B-1)} we can see that the Torelli property does not hold for $(D_a,\{p_0,p_1,p_2\})$.

%%%%%%%%%%%%%%%%%%%%%%%%%%%%%%%%%%%%%%

\section{Generalized logarithmic shaves with fixed  points outside the divisor}\label{sec-blowup}

In the previous section we were studying the Torelli property for the generalized logarithmic sheaf $\Omega_{\PP^2}^1(\log (C,Z))$ associated to curves $C\subset \PP^2$ and set of points $Z$ contained in $C$. In particular we payed particular attention to the case of lines, conics (Subsection \ref{sec-planeconics}) and cubic curves (Subsection \ref{sec-planecubics}). In this section we are going to address the issues raised in Subsection \ref{subsec-pointsout} for a similar set of cases to underline the similarities and differences between the setting in Subsection \ref{subsec-pointsin} (blowing up a set of points contained in the curves under consideration) and Subsection \ref{subsec-pointsout} (blowing up a set of points with empty intersection with the curves under consideration). Namely, we are going to study the Torelli property for linear systems associated to the strict transforms of lines (Subsection \ref{subsec-cubics}), conics (Subsection \ref{subsec-sextic}) and cubic curves (Subsection \ref{subsec-deg9}) on surfaces $S$ obtained by blowing-up points outside of the curves under consideration. Indeed, thanks to Remark \ref{composition-pushforward} we decided to place ourselves on the classical case $S\subset \PP^3$ a smooth cubic surface.

In order to state better the problems that we are going to face, let us start recalling some classical facts. Associated to a disjoint set of six lines $E_i$ in $S$ let $\pi:S\rightarrow\PP^2$ be the blow-up morphism of $\PP^2$ along the set of six general points $p_i=\pi(L_i)$.

Recall that the exceptional divisors $E_i:=\pi^{-1}(p_i)$ associated to the points $p_i$ together with $L$ a divisor class in $\pi^*\Oo_{\PP^2}(1)$, freely generate $\mathrm{Pic}(S)\cong \ZZ \langle L, E_1, \dots, E_6\rangle \cong \ZZ^{\oplus 7}$ with intersection numbers
\[
L^2=1,\: L.E_i=0,\: E_i.E_j=0 \:\mbox{ and }\: E_i^2=-1
\]
for each $i\ne j$. The anticanonical line bundle $\Oo_S(-K_S) \cong \Oo_S(3L-\sum_{i=1}^6 E_i)$ turns out to be very ample so that it induces an embedding $\iota: S \hookrightarrow \PP^3$ as a smooth cubic surface. It is classically known that there are exactly $27$ lines in $S\subset \PP^3$:
\begin{itemize}\label{popo}
\item [(i)] $E_i$ for $1\le i \le 6$;
\item [(ii)] $L_{ij}$ the unique element in $|\Oo_S(L-E_i-E_j)|$ for $1\le i<j\le 6$;
\item [(iii)] $\widehat{L}_i$ the unique element in $|\Oo_S(2L+E_i-\sum_{k=1}^6E_k)|$ for $1\le i \le 6$.
\end{itemize}
\medskip
Furthermore, recall the following.
\begin{remark}\label{remmm}
Denoting by $\pi : \widetilde{\PP}^2\rightarrow \PP^2$ the blow up of $\PP^2$ at $k$ distinct points $p_1, \dots, p_k$, let $L$ be a divisor class in $\pi^*\Oo_{\PP^2}(1)$ and $E_i$ be the exceptional divisor over $p_i$ for each $i$. Then, for a line bundle $\Ll=\Oo_{\widetilde{\PP}^2}\left(aL+\sum_{i=1}^{k}b_i E_i\right)$ on $\widetilde{\PP}^2$ with $a,b_i \in \ZZ$, we have
\begin{itemize}
\item [(i)] $\pi_* \Ll\cong\Oo_{\PP^2}(a)\otimes \left(\bigotimes_{b_i<0}(\Ii_{p_i, \PP^2})^{ (-b_i)}\right)$.
\item [(ii)] $\mathbf{R}^1\pi_*\Ll \cong \bigoplus_{b_i>2}\left(\Oo_{\PP^2}\big/\Ii_{p_i}^{b_i-2}\right)$.
\end{itemize}
\noindent Indeed, we get (i) from the projection formula and the fact that $\pi_*\Oo_{\widetilde{\PP}^2}(tE_i) \cong \Oo_{\PP^2}$ for $t\ge 0$ or $(\Ii_{p_i, \PP^2})^{\otimes (-t)}$ for $t<0$. The part (ii) is obtained by applying the Theorem on Formal Function in \cite[III.11]{Hartshorne} to the line bundle $\Ll$; refer to \cite[Proposition 1.3.7]{Coronica}. 
\end{remark}

In this setting, we are going to study the Torelli property for different effective linear systems $|\Oo_S(D)|$ and its relationship with the Torelli property for 
$\FF(Z',\pi(D))$ on $\PP^2$ where $Z'$ is the intersection of $\pi(D)$ with $Z:=\{p_1,\dots,p_6\} $. In particular, we will now pay our attention to the following effective linear systems, which are analogous to the ones considered in the previous section.

Firstly, in Subsection \ref{subsec-cubics} we will study the logarithmic sheaf associated to a divisor from the linear system $|\Oo_S(L)|$ defined as strict transforms of lines in $\PP^2$. We will notice that, through a Cremona transformation, the problem is equivalent to considering the divisors $D \in  |\Oo_S(2L -E_1 - E_2 - E_3)|$, namely, a cubic curve $D$ in the surface $S$ which is projected on a conic $C$ in $\PP^2$ passing through the three fixed points $Z'=\{p_1,p_2,p_3\}$. Recall that the Torelli property holds for $(C,Z')$  (see Theorem \ref{torr-k}) and hence, by Lemma \ref{torr-proj}, it turns out that the Torelli property also holds for $D$ in $S$.

Next, in Subsection \ref{subsec-sextic} we will consider the logarithmic sheaf associated to a divisor $D \in |\Oo_S(2L)|$. The reader should confront it with the results from Subsection \ref{sec-planeconics}: the projection $\pi (D)$ of $D$ to $\PP^2$ is a conic which does not pass through any blown-up point, i.e., a generic one. We have already noticed that $\pi_* \Omega_S^1(\log D) \cong \Omega_{\PP^2}^1(\log \pi(D))$ and the Torelli property does not hold for $|\Oo_{\PP^2}(\pi(D))|$. However, we will see that the Torelli property holds for $|\Oo_S(2L)|$ in the cubic surface $S$.

Finally, in Subsection \ref{subsec-deg9} we are going to show that the Torelli property holds for the linear system $|\Oo_S(3L)|$.

In order to perform some computations throughout this section, we need to recall the behaviour of the restriction of the cotangent bundle $\Omega_S^1$ to different type of curves.

\begin{lemma}
We have:
\begin{itemize}
\item  $\left(\Omega_S^1\right)_{|C} \cong \Oo_{\PP^1}(-2)\oplus \Oo_{\PP^1}(1)$ for any line $C\subset S$;
\item  $\left(\Omega_S^1\right)_{|C} \cong \Oo_{\PP^1}(-2)\oplus \Oo_{\PP^1}(-1)$ for any smooth $C \in |\Oo_S(L)|$;
\item  $\left(\Omega_S^1\right)_{|C} \cong \Oo_{{\PP^1}}(-2)\oplus \Oo_{{\PP^1}}$ for any smooth $ C\in  |\Oo_S(L-E_i)|$, with $i=1,\ldots,6$;
\item  $\left(\Omega_S^1\right)_{|C} \cong \Oo_{\PP^1}(-2)\oplus \Oo_{\PP^1}(-4 + |I|)$ for any smooth $C \in |\Oo_S(2L-\sum_{i \in I}E_i)|$, with $I \subset \{1\ldots,6\}, 1 \leq |I| \leq 4$.
\end{itemize}
\end{lemma}
\begin{proof}
Consider the relative cotangent short exact sequence
\[
0 \to \Oo_S(-C)_{|C} \to \left(\Omega_S^1\right)_{|C} \to \Omega_C^1 \to 0,
\]
from which follow all the stated restrictions, because $C$ is rational and so $\Omega_C^1 \cong \Oo_{\PP^1}(-2)$.
\end{proof}

\subsection{Cubic curves in $|\Oo_S(2L -E_1 - E_2 - E_3)|$}\label{subsec-cubics}
Recalling that, by Lemma \ref{torr-proj} and Theorem \ref{torr-k}, these divisors satisfy the Torelli property, we will observe that the considered class is equivalent to the class $|\Oo_S(L)|$ after a Cremona transformation. Therefore we will study the stability property for the logarithmic bundles associated to the latter divisors and describe the injective moduli map induced by  the Torelli property.

\begin{remark}\label{rem998}
Consider the following new set of $6$ divisors
\[
\begin{array}{cccl}
\tilde{E}_1 :=L_{23},&\tilde{E}_2:=L_{13},&\tilde{E}_3:=L_{12}\\
\tilde{E}_4:=E_4,&\tilde{E}_5:=E_5,&\tilde{E}_6:=E_6,
\end{array}
\]
which satisfies $\tilde{L}^2 = 1, \tilde{E}_i^2=-1, \tilde{L}.\tilde{E}_i=0$ for $\tilde{L}:=D\in |\Oo_S(2L-E_1-E_2-E_3)|$. Therefore they represent the set of exceptional divisors coming from a different blow-up of $6$ points in $\PP^2$, say $\tilde{\pi} : S \rightarrow \PP^2$, defined by the linear system $|\Oo_S(\tilde{L})|$. This gives a Cremona transformation:
\[
\begin{tikzcd}[
 column sep=small,
   ar symbol/.style = {draw=none,"\textstyle#1" description,sloped},
  isomorphic/.style = {ar symbol={\cong}},
  ]
 &S\ar[ld, "\pi~", labels=above] \ar[rd, "\tilde{\pi}"]\\
 \PP^2 \ar[rr, dashed] && \PP^2.
 \end{tikzcd}
\]
Noticing that 
\[
\tilde{H}:= 3\tilde{L}-\sum_{i=1}^6\tilde{E}_i = 3L-\sum_{i=1}^6 E_i = H,
\]
we are not considering only the same surface, but the same embedding as well. In particular, the study on the logarithmic vector bundle $\Omega_S^1(\log D)$ can be replaced by the study on $\Omega_S^1(\log D')$ with smooth $D' \in |\Oo_S(L)|$. On the other hand, we get that $\pi_*\Omega_S^1(\log D') \cong \Omega_{\PP^2}^1\left(\log \pi(D')\right) \cong \Oo_{\PP^2}(-1)^{\oplus 2}$.   
\end{remark}

To conclude this part, we show that the logarithmic vector bundles $\Omega_S^1(\log D)$, with $D\in |\Oo_S(L)|$, are stable with respect to the hyperplane section. 

\begin{theorem}\label{stab1}
Given a twisted cubic curve $D\in |\Oo_S(L)|$, the logarithmic bundle $\Omega_S^1(\log D)$ is a $\mu$-stable vector bundle with respect to $\Oo_S(H)$ and with Chern classes $(c_1, c_2)=\left(K_S+L, 7\right)$. Therefore, the induced rational map 
\[
\Psi_1 : \PP \mathrm{H}^0(\Oo_S(L))\cong \PP^2 \dashrightarrow \mathbf{M}_S^H(K_S+L, 7)
\]
is generically one-to-one.  
\end{theorem}

The proof of Theorem \ref{stab1}, as well as the proof of Theorem \ref{stab2}, will follow the subsequent strategy. We suppose the existence of a $\mu$-destabilizing sheaf of rank one with respect to the polarization $H$, which can be assumed to be a line bundle; this follows from \cite[Theorem II.1.2.2]{OSS}. This provides us a list of possible destabilizing line bundles, listed in terms of their first Chern class. Using several techniques, from geometric properties of the curve to the restriction on a rational curve provided by the Key Restriction Lemma \ref{keylem}, we show that all the items of the list lead to contradiction, proving therefore the stability of the logarithmic vector bundle.

\begin{proof}
Assume that $\Omega_S^1(\log D)$ is not $\mu$-stable with respect to $\Oo_S(H)$. Considering a destabilizing line bundle 
\[
\Oo_S\left(aL+\sum_{i=1}^{6}b_i E_i\right) \hookrightarrow \Omega_S^1(\log D),
\]
we have
\begin{equation}\label{deg-ineq}
3a + \sum_{i=1}^{6}b_i= \mu\left(\Oo_S(aL+\sum_{i=1}^{6}b_i E_i)\right) \geq \mu\left(\Omega_S^1(\log D)\right) = 0
\end{equation}
Thanks to Lemma \ref{keylem}, we know that for a generic element $C \in |\Oo_S(2L-E_1-E_2-E_3-E_4)|$, we have that $\left(\Omega_S^1(\log D)\right)_{|C} \simeq \Oo_{\PP^1}^{\oplus 2}$. This implies that the restriction of $\Oo_S(aL+\sum_{i=1}^{6}b_i E_i)$ to $C$ must also have non-positive degree, i.e., $2a+b_1+b_2+b_3+b_4 \leq 0$. By analogous computation we obtain that $2a+b_1+b_2+b_5+b_6 \leq 0$ and $2a+b_3+b_4+b_5+b_6 \leq 0$. Summing up the three inequalities together with (\ref{deg-ineq}), we get that 
\begin{equation}\label{deg-equality}
3a + \sum_{i=1}^{6}b_i = 0
\end{equation}
Again by Lemma \ref{keylem}, we know that for the generic element $C \in |\Oo_S(L-E_i-E_j)|$, with $1\leq i < j \leq 6$, we have that $\left(\Omega_S^1(\log D)\right)_{|C} \simeq \Oo_{\PP^1}(-1) \oplus \Oo_{\PP^1}(1)$. Restricting the destabilizing bundle to $C$, we get
\[
a + b_i + b_j \leq 1 \mbox{ for any } 1\leq i < j \leq 6.
\]
Moreover, from
\[
 -a -b_5-b_6 = 2a+b_1+b_2+b_3+b_4 \leq 0
\]
we get that $a+b_5+b_6 \geq 0$ as well, and more in general
\begin{equation}\label{ineq32}
a+b_i+b_j \geq 0 \mbox{ for any } 1\leq i < j \leq 6.
\end{equation}
But this implies, combined with the equality in (\ref{deg-equality}), that the inequalities in (\ref{ineq32}) are all actually equalities. Thus there exists a positive integer $k$ such that 
\[
a=-2k, \mbox { and } ~b_i = k \mbox{ for any } 1\leq i \leq 6,
\]
i.e., any destabilizing sheaf is of the form $\Oo_S(-2kL + \sum_{i=1}^6k E_i)$ for a positive integer $k$. Now consider the line $\widehat{L}_i$, the unique line in the system $|\Oo_S(2L+E_i-\sum_{j=1}^6E_j)|$. By Lemma \ref{keylem} we have that $\left(\Omega_S^1(\log D)\right)_{|\widehat{L}_i} \simeq \Oo_{\PP^1} \oplus \Oo_{\PP^1}(1)$. Restricting the destabilizing bundle, we also get that
\[
\left(\Oo_S(-2kL+ \sum_{i=1}^6k E_i)\right)_{|\widehat{L}_i} \simeq \Oo_{\PP^1}(k)
\]
and this implies $k\leq 1$, and so $k\in \{0,1\}$. Assume that $k=0$ and apply the push-forward functor $\pi_*$ to the injection $\Oo_S \hookrightarrow \Omega_S^1(\log D)$. Then we get an injection 
\[
\Oo_{\PP^2} \simeq \pi_* \Oo_S \hookrightarrow \Omega_S^1(\log D) \simeq \Oo_{\PP^2}(-1)^2,
\]
where the last isomorphism is due to Remark \ref{rem998}, and this gives a contradiction. Finally assume that $k=1$ and so we have $\Oo_S(-2L+\sum_{i=1}^6 E_i)$ as the only possible destabilizing bundle. Compositing it with the Poincar\'e residue map for $\Omega_S^1(\log D)$, we obtain a nontrivial map
\[
\Oo_S(-2L+\sum_{i=1}^6 E_i) \rightarrow \Oo_{D},
\]
unless the destabilizing bundle maps into $\Omega_S^1$, contradicting its stability with respect to $\Oo_S(H)$; see \cite{Fah}. From the isomorphism 
\[\mathrm{Hom}_{\Oo_S}\left(\Oo_S(-2L+\sum_{i=1}^6 E_i), \Oo_{D}\right) \cong \mathrm{Hom}_{\Oo_{D}}\left(\Oo_{D}(-2L+\sum_{i=1}^6 E_i), \Oo_{D}\right),
\]
this nontrivial map factors through the sheaf $ \Oo_S(-2L+\sum_{i=1}^6 E_i)_{|D}$, and so we obtain the following commutative diagram
\[
\begin{tikzcd}[
  row sep=small, column sep=small,
  ar symbol/.style = {draw=none,"\textstyle#1" description,sloped},
  isomorphic/.style = {ar symbol={\cong}},
  ]
&&&0\ar[d]   \\
&&0 \ar[d] &\Oo_S(-3L+\sum_{i=1}^6 E_i) \ar[d] \\
&&\Oo_S(-2L+\sum_{i=1}^6 E_i) \ar[r, isomorphic]\ar[d] & \Oo_S(-2L+\sum_{i=1}^6 E_i) \ar[d]  \\
0 \ar[r] & \Omega_{S}^1 \ar[r] & \Omega_S^1(\log D) \ar[r] & \Oo_{D} \ar[r]  & 0,\\ 
\end{tikzcd}
\]
from which we obtain a nonzero map $\Oo_S(-3L+\sum_{i=1}^6 E_i) \cong \Oo_S(-H) \hookrightarrow \Omega_S^1$. On the other hand, by straightforward computation 
 we have $\mathrm{H}^0\left(\left(\Omega^1_{\PP^3}(1)\right)_{|S}\right) = \mathrm{H}^1\left(\Oo_S(-2H)\right) = 0$ and, from the following canonical short exact sequence, 
\[
0 \rightarrow \Oo_S(-2H) \rightarrow \left(\Omega^1_{\PP^3}(1)\right)_{|S} \rightarrow \Omega^1_S(H) \rightarrow 0,
\]
we get that $\mathrm{H}^0\left(\Omega^1_S(H)\right) =0$, giving a contradiction. Therefore $\Omega_S^1(\log D)$ is stable. 
\end{proof}

Let us now describe the space of irreducible cubics in $|\Oo_S(L)|$. Indeed, the map $\Psi_1$ is defined on such locus and the description of its complementary provides a comprehension of the indeterminacy locus of $\Psi_1$ itself.

\begin{lemma}\label{llem1}
There exist six curves $C_1, \dots, C_6\subset |\Oo_S(L)|\cong \PP^2$ such that the points in the complement
\[
|\Oo_S(L)|\setminus \left(\cup_{i=1}^6 C_i \right)
\]
correspond to the twisted cubics in the system. The curves $C_i$'s satisfy the following:
\begin{itemize}
\item each curve $C_i$ is isomorphic to $\PP^1$;
\item For $i\ne j$, the intersection $C_i\cap C_j$ is a single point corresponding to the union of three lines $L_{ij} \cup E_i \cup E_j$.
\end{itemize}
\end{lemma}
\begin{proof}
We have the rational map $\phi:\PP^2\dasharrow S\subset\PP^3$ given by four cubic forms on $\PP^2$ vanishing at $Z$. If $\ell \in |\Oo_{\PP^2}(1)|\cong|\Oo_S(L)|$ is a line, not passing through any of the $p_i$, then $\phi_{| \ell}$ is an isomorphism and $\phi(\ell)$ is a twisted cubic curve. If $\ell$ passes through one of the points, say $p_i$, then the corresponding curve in $S$ is $L'\cup E_i$ with $L'\in |\Oo_S(L-E_i)|\cong\PP^1$. Here, $\bar{\phi(\ell\setminus \{p_i\})}$ is an irreducible conic, except for the union of the two lines $L'=L_{ij}\cup E_j$ when $\ell$ also passes through $p_j$.
\end{proof}

\subsection{Sextic curves in $|\Oo_S(2L)|$}\label{subsec-sextic}
Consider $D \in |\Oo_S(2L)|$, whose projection is a conic $C$ in $\PP^2$ not passing through any blown-up point and hence the Torelli property does not hold for $C$. However, we will prove that the Torelli property holds for $D$ in the cubic surface and hence we prove that the Torelli property can be achieved after blow-up.

\begin{proposition}\label{quad}
Let $D_1$ and $D_2$ be any two distinct elements in the linear class $|\Oo_S(2L)|$. Then
\[
\Omega_S^1(\log D_1) \not\cong \Omega_S^1(\log D_2),
\]
i.e., the Torelli property holds for the divisors in $|\Oo_S(2L)|$.
\end{proposition}

\begin{proof}
Consider the rational curve $D\in |\Oo_S(2L)|$, of degree 6. Let $Q = \pi(D)\subset \PP^2$ be the conic coming from the projection. Denote the $6$ blown-up points by $Z = \{p_1, \dots, p_6\}$, therefore $Z\cap Q =\emptyset$. Now pick a smooth conic $i: C\hookrightarrow S$ from the complete linear system $|\Oo_S(L-E_1)|$ so that $\pi(C)$ is a line through $p_1$. Set $C \cap D=\{q_1, q_2\}$; notice that $q_1=q_2$ when the line $\pi(C)$ is tangent to the conic $Q$ at $\pi(q_1)$. If $q_1\ne q_2$, the upper horizontal sequence of the diagram (\ref{keydiag}) gives
\[
0\to \Tt_C\left (-\log \{q_1, q_2\}\right) \to i^* \Tt_S(-\log D) \to \Oo_C \to 0
\] 
so that we have $\left( \Tt_S(-\log D)\right)_{|C} \cong \Oo_C^{\oplus 2}$. On the other hand, if $q_1=q_2$, the diagram produces an exact sequence
\[
0\to \Tt_C\left (-\log \{q_1\}\right) \to i^* \Tt_S(-\log D) \to \Oo_C (-q_1)\to 0
\] 
so that we have $\left( \Tt_S(-\log D)\right)_{|C} \cong \Oo_C\left(-q_1\right)\oplus \Oo_C\left(q_1\right)$. Thus to the sextic $D\in |\Oo_S(2L)|$ we associate two (tangent) lines to the conic $Q=\pi(D)$ passing through $p_1$, as its jumping lines. By replacing $p_1$ by the other points in $Z$, we get $12$ jumping lines, tangent to $Q$ and this uniquely determines $Q$ and $D$ as well. Hence, $D$ can be recovered from $\Omega_S^1(\log D)$. 
\end{proof}

Now, using an analogous argument as in Proposition \ref{stab1}, we can show the stability of logarithmic vector bundles associated to the curves $D\in |\Oo_S(2L)|$ and induce an analogous moduli map. To do so, let us explicit the following geometry property of conics in the projective plane. 

\begin{remark}\label{remquad}
Consider the dual plane $(\PP^2)^{\vee}$ and its Veronese embedding $\iota : (\PP^2)^{\vee} \hookrightarrow V \subset (\PP^5)^{\vee}$, where $(\PP^5)^{\vee}$ parametrizes the conics in $\PP^2$. Fix a smooth conic $C\subset \PP^2$. Then the double lines in $\PP^2$ tangent to $C$ forms a quartic curve $\iota(C^{\vee})$, where $C^{\vee}$ is the dual conic of $C$. We say that a conic $C'$ is {\it tetra-tangent} to $C$ if the intersection $C'\cap C$ supports only one point. Then the set $T_4(C)\subset (\PP^5)^{\vee}$ of tetra-tangent conics to $C$ is a cone over $\iota(C^{\vee})$ with the vertex point $[C]\in (\PP^5)^{\vee}$. For two points $Z=\{p_1, p_2\}$ with $Z\cap C=\emptyset$, let $H_i$ be the subspace of conics passing through $p_i$. Then the intersection $T_4(C) \cap H_1 \cap H_2$ consists of $4$ points. In particular, there exist no $4$ smooth conics $C_1, \dots, C_4 \in H_1 \cap H_2$ that are tetra-tangent to a smooth conic $C$, when $C$ is tangent to the line $p_1p_2$. Indeed, the singular conic supporting the line $p_1p_2$ would be another conic tetra-tangent to $C$, giving a contradiction. 

On the other hand, for any smooth tetra-tangent conic $C'$ to $C$, there is no triangle $q_1q_2q_3$ that is inscribed in $C'$ and circumscribing $C$, i.e., $C$ is not Poncelet related to any smooth tetra-tangent conics to $C$. Without loss of generality, set $C=V(2x_0x_1+x_2^2)$ and assume that $C'$ is contained in the ruling over the double line $V(x_0^2)$, i.e., $C'=(2x_0x_1+x_2^2+\alpha x_0^2)$ for some $\alpha \ne 0$. Setting $A$ and $A'$ be the corresponding symmetric matrices to $C$ and $C'$, we have
\[
\det (t_0A+t_1A')= \det \begin{pmatrix}  t_1 \alpha & t_0+t_1 & 0 \\ t_0+t_1 & 0 & 0 \\ 0 &0 & t_0+t_1 \end{pmatrix}=-(t_0+t_1)^3
\]
Then we get the assertion by \cite[Theorem 2.3.14]{D}; in the notation of \cite[Theorem 2.3.14]{D} we have $(\Theta, \Theta', \triangle')=(-3,-3,-1)$ so that we have ${\Theta'}^2-4\Theta\triangle'=-3\ne 0$.
\end{remark}

\begin{theorem}\label{stab2}
Given a smooth sextic curve $D\in |\Oo_S(2L)|$ the logarithmic bundle $\Omega_S^1(\log D)$ is $\mu$-stable  with respect to $\Oo_S(H)$ and with Chern classes $(c_1, c_2)=\left(K_S+2L, 7\right)$. Therefore the induced rational map 
\[
\Psi _2: \PP \mathrm{H}^0(\Oo_S(2L))\cong \PP^5 \dashrightarrow \mathbf{M}_S^H(K_S+2L, 7)
\]
is generically one-to-one.  
\end{theorem}

\begin{proof}
Assume that $\Omega_S^1(\log D)$ is not $\mu$-stable with respect to $\Oo_S(H)$. Considering a destabilizing line bundle 
\[
\Ll:=\Oo_S\left(aL+\sum_{i=1}^{6}b_i E_i\right) \hookrightarrow \Omega_S^1(\log D),
\]
we have
\begin{equation}\label{deg-ineq2}
3a + \sum_{i=1}^{6}b_i= \mu\left(\Oo_S(aL+\sum_{i=1}^{6}b_i E_i)\right) \geq \mu\left(\Omega_S^1(\log D)\right) = \frac{3}{2}.
\end{equation}
Using Lemma \ref{keylem} as in Proposition \ref{stab1}, we get the following inequalities, where the left column stands for the linear systems that the curve $C$ belongs to: note that $C$ is a rational curve to which the destabilizing line bundle is restricted. 
\begin{center}

\begin{tabular}{ c|c }
  \hline
  Linear systems &Inequalities \\
  \hline
$L$ & $a\le 0$ \\
 $E_i$ & $b_i \ge -1$ \\
 $L-E_i$ & $a+b_i \le 0$\\
 $L-E_i-E_j$ & $a+b_i+b_j \le 2$\\
   \hline
\end{tabular}
\captionof{table}{}\label{tab}
\end{center}
Note also that the last inequality can become an equality, only when $C$ is tangent to $D$. Now consider the line $C=\widehat{L}_i$ from $|\Oo_S(2L+E_i-\sum_{j=1}^6 E_j)|$.

\quad (a) Assume first that the intersection of the line $\widehat{L}_i$ with $D$ supports at least three points for each $i$, i.e., either $\widehat{L}_i$ is not tangent to $D$ or $\widehat{L}_i \cap D=\{2q_1, q_2, q_3\}$ with three distinct points $q_1, q_2, q_3$. Applying Lemma \ref{keylem}, we get $2a-b_i +\sum_{j=1}^6 b_j \le 2$, from which we get 
\[
\frac{3}{2}\le 3a+\sum_{j=1}^6 b_j \le 2+(a+b_i) \le 2.
\]
This implies that $2a-b_i +\sum_{j=1}^6 b_j =2$ and $a+b_i=0$ for each $i$. Then the only possibility would be $(a, b_1, \dots, b_6)=(0,\dots, 0)$, which contradicts the inequality (\ref{deg-ineq2}).

\quad (b) Now assume that the intersection $\widehat{L}_i \cap D$ supports two points for some $i$, say $i=6$, but not one point for any $i$. Then, applying Lemma \ref{keylem} to $\widehat{L}_6$, we get $2a+\sum_{j=1}^5 b_j \le 3$ and so 
\[
2\le 3a+\sum_{j=1}^6b_j=(a+b_6)+(2a+\sum_{j=1}^5 b_j) \le 3
\]
would imply that $b_6\in \{-a,-a-1\}$, due to $a+b_6\le 0$. Thus, in general, combining the results in (a), we get
\[
b_i \in \{-a,-a-1\}
\]
for any $i=1,\dots, 6$. Calling $\delta$ the number of $i$'s with $b_i=-a-1$, we get
\[
2\le -3a-\delta=3a+\sum_{j=1}^6 b_j \le 3,
\]
implying $-3 \le a \le -1$. If $a=-3$, we get that $b_i=2$ for each $i$, i.e., $\delta=6$, contradicting the inequality $2a+\sum_{j=1}^5 b_j \le 3$. If $a=-2$, we get $\delta \in \{3,4\}$. Since we have $-b_i + \sum_{j=1}^6 b_j \le 7$ for any $i$, the only possibility is that $\delta =4$. In particular, in addition to $\widehat{L}_6$, there would exist three more lines $\widehat{L}_{i_1}, \dots, \widehat{L}_{i_3}$ with $1 \le i_1<\dots < i_3 \le 5$ such that each intersects $D$ in at most two points. Without loss of generality, set $(b_1, \dots, b_6)=(2,2,1,1,1,1)$ so that each of $\widehat{L}_3, \dots, \widehat{L}_6$ intersects $D$ at most two points, and $L_{12}$ is tangent to $D$. Then we obtain an exact sequence
\[
0\to \Ll \to \Omega_S^1(\log D) \to \Ii_{\Gamma, S}(L-E_1-E_2) \to 0,
\]
for a zero-dimensional subscheme $\Gamma\subset S$ of length $5$, whose restriction to the exceptional divisor $E_i$ with $i=3,4,5,6$ gives
\[
0\to \Oo_{E_i}(-1+c) \to \Oo_{E_i}(1)\oplus \Oo_{E_i}(-2) \to \Oo_{E_i}(-c) \to 0
\]
for $c=\mathrm{length}(E_i \cap D)$, because $\left( \Omega_S^1(\log D)\right)_{|E_i} \cong \left( \Omega_S^1\right)_{|E_i} $. Then we have $c=2$, i.e., each $E_i$ contains exactly two points of $\Gamma$, and this is impossible. 

Now we assume the remaining case that $a=-1$ and so $\delta\in \{0,1\}$, i.e., the possible destabilizing line bundle is either
\[
\Oo_S\left(-L+ \sum_{j=1}^6 E_j\right) \quad {\text or} \quad \Oo_S\left(-L-E_i+\sum_{j=1}^6 E_j \right).
\]
for some $i$. In the former case, as in the last part of the proof of Proposition \ref{stab1} we get a commutative diagram
\[
\begin{tikzcd}[
  row sep=small, column sep=small,
  ar symbol/.style = {draw=none,"\textstyle#1" description,sloped},
  isomorphic/.style = {ar symbol={\cong}},
  ]
&&&0\ar[d]   \\
&&0 \ar[d] &\Oo_S(-3L+\sum_{i=1}^6 E_i) \ar[d] \\
&&\Oo_S(-L+\sum_{i=1}^6 E_i) \ar[r, isomorphic]\ar[d] & \Oo_S(-L+\sum_{i=1}^6 E_i) \ar[d]  \\
0 \ar[r] & \Omega_{S}^1 \ar[r] & \Omega_S^1(\log D) \ar[r] & \Oo_{D} \ar[r]  & 0,\\ 
\end{tikzcd}
\]
from which we obtain a nonzero map $\Oo_S(-3L+\sum_{i=1}^6 E_i) \cong \Oo_S(-H) \hookrightarrow \Omega_S^1$, contradicting the vanishing $\mathrm{H}^0(\Omega_S^1(H))=0$. In the latter case, we similarly obtain a nonzero map 
\[
\Oo_S\left ( -3L-E_i+\sum_{j=1}^6 E_j \right)\hookrightarrow \Omega_S^1.
\]
On the other hand, we have $\mathrm{H}^0(\Omega_S^1(H+E_i))=0$ by an analogous computation as at the end of the proof of Theorem \ref{stab1}, giving a contradiction. 

\quad (c) Finally we assume that $\widehat{L}_i \cap D$ supports a point for some $i$, say $i=6$, i.e., $\widehat{L}_6$ is quad-tangent to $D$. Then we get an inequality $2a+\sum_{j=1}^5b_j\le 4$, and as in (b) we get $-5 \le a \le -1$ with $b_i \in \{-a, -a-1, -a-2\}$ for each $i$. Note that in general we have 
\begin{equation}\label{ine343}
2a-b_i+\sum_{j=1}^6 b_j \le 4
\end{equation}
for any $i$. Setting $\delta_t$ to be the number of $i$'s with $b_i=-a-t$ for $t\in \{0,1,2\}$, we have
\[
2\le 3a+\sum_{j=1}^6 b_j= -3a-\delta_1-2\delta_2\le 4.
\]
If $a=-5$, then we have $11\le \delta_1+2 \delta_2 \le 13$, implying that $(\delta_1, \delta_2) \in \{(0,6), (1,5)\}$. In each case we have a contradiction to the inequality (\ref{ine343}). Similarly, we can obtain the list of possible pairs $(\delta_1, \delta_2)$ together with values of $b_i$'s (see Table \ref{tab2}). 
\begin{center}
\begin{tabular}{ c|c|c }
  \hline
 $a$ &$(\delta_1, \delta_2)$ & $\{b_1, \dots, b_6\}$ \\
  \hline
$a=-4$ & $(0,5)$ & $\{4,2,2,2,2,2\}$\\
            & $(2,4)$ & $\{3,3,2,2,2,2\}$ \\
  \hline
$a=-3$ & $(1,3)$ & $\{3,3,2,1,1,1\}$\\
            & $(3,2)$ & $\{3,2,2,2,1,1\}$\\
            & $(5,1)$ & $\{2,2,2,2,2,1\}$\\
            & $(6,0)$ & $\{2,2,2,2,2,2\}$\\
  \hline
  $a=-2$ & $(2,1)$ & $\{2,2,2,1,1,0\}$\\
               & $(3,0)$ & $\{2,2,2,1,1,1\}$\\
               & $(4,0)$ & $\{2,2,1,1,1,1\}$\\
  \hline
  $a=-1$ & $(0,0)$ & $\{1,1,1,1,1,1\}$\\
              & $(1,0)$ & $\{1,1,1,1,1,0\}$\\
   \hline
\end{tabular}
\captionof{table}{}\label{tab2}
\end{center}
Recall that we have the following two conditions satisfied. 
\begin{itemize}
\item If $a+b_i+b_j =2$, then the line $\pi(L_{ij})$ is tangent to $\pi(D)$. 
\item If $2a-b_i+\sum_{j=1}^6b_j = 4$, then the conic $\pi(\widehat{L}_i)$ is tetra-tangent to $\pi(D)$. 
\end{itemize}
Note that the last three cases in Table \ref{tab2} are impossible by the last argument in (b). Assume that $(a,\delta_1, \delta_2)=(-4,0,5)$ and set $b_1=4$. Then the five lines $L_{12}, L_{13}, \dots, L_{16}$ would be tangent to $D$. It is impossible, because the six points $p_1, \dots, p_6$ are in general position. The case $(a,\delta_1, \delta_2)=(-3,1,3)$ is not possible, because of the inequality $a+b_i+b_j\le 2$ in Table \ref{tab}. Also by Remark \ref{remquad} the cases 
\[
(a,\delta_1, \delta_2)\in \left\{(-4,2,4), (-2,2,1), (-2,3,0), \right\}
\]
are not possible. In case $(a,\delta_1, \delta_2)=(-3,3,2)$, assuming $(b_1, \dots, b_6)=(3,2,2,2,1,1)$, the lines $L_{12}, L_{13}, L_{14}$ would be tangent to $D$. Then three points $p_1$ and two of $p_2, p_3, p_4$ are collinear, a contradiction. Thus there are two cases left; $(a, \delta_1, \delta_2) \in \{(-3,5,1), (-3, 6,0)\}$, and we may assume 
\[
(a, b_1, \dots, b_6)\in \{(-3,2,2,2,2,2,1), (-3,2,2,2,2,2,2)\}.
\]
In case of $(a, b_1, \dots, b_6)=(-3,2,2,2,2,2,2)$, we may consider an exact sequence
\begin{equation}\label{eqa551}
0\to \Oo_S\left(-3L+\sum_{j=1}^6 2E_j\right) \to \Omega_S^1(\log D) \to \Ii_{p, S}\left(2L-\sum_{j=1}^6 E_j\right) \to 0,
\end{equation}
where $p$ is a point in $S$. Restricting the sequence (\ref{eqa551}) to $E_j$, we get that $p\not \in E_j$ for any $j$, because $\left(\Omega_S^1(\log D) \right)_{|E_j} \cong \left(\Omega_S^1\right)_{E_j} \cong \Oo_{E_j}(-2)\oplus \Oo_{E_j}(1)$. Similarly, if we restrict the sequence (\ref{eqa551}) to $\widehat{L}_j$, we get
\[
\left(\Omega_S^1(\log D)\right)_{|\widehat{L}_j} \cong \Oo_{\widehat{L}_j}(4+c) \oplus \Oo_{\widehat{L}_j}(-1-c),
\]
where $c=1$ if $p\in \widehat{L}_j$, or $c=0$ otherwise. On the other hand, applying Lemma \ref{keylem}, we also obtain 
\[
\left(\Omega_S^1(\log D)\right)_{|\widehat{L}_j} \cong \Oo_{\widehat{L}_j}(k-2) \oplus \Oo_{\widehat{L}_j}(5-k),
\]
where $k$ is the number of distinct points in $\mathrm{Supp}(D \cap \widehat{L}_j)$. Thus we get that $(c,k)=(0,1)$, and in particular $p$ is not contained in $\widehat{L}_j$ for any $j$. Now pick a unique smooth curve $C=C_{1234} \in |\Oo_S(2L-E_1-E_2-E_3-E_4)|$, also passing through $p$. Then we get 
\[
\left(\Omega_S^1(\log D)\right)_{|C} \cong \Oo_{C}(3) \oplus \Oo_{C}(-1)
\]
from the sequence (\ref{eqa551}), while we also get an exact sequence 
\[
0\to \Oo_{C}(4-k) \to \left(\Omega_S^1(\log D)\right)_{|C} \to \Oo_C(k-2) \to 0
\]
from Lemma \ref{keylem}, where $k$ is the number of distinct points in $\mathrm{Supp}(D \cap C)$. So the only possibility is $k=1$, i.e., $C$ is tetra-tangent to $D$. Similarly we can obtain $6\choose 4 = 15$ smooth curves $C_{i_1i_2i_3i_4}$ that are tetra-tangent to $D$, all passing through $p$, for any $1\le i_1<i_2<i_3<i_4\le 6$. In particular, in the $3$-dimensional family of conics through $p$ and $p_1$, we have the following $10$ tetra-tangent conics to $D$:
\[
C_{1234}, C_{1235}, C_{1236}, C_{1245},C_{1246}, C_{1256}, C_{1345}, C_{1346}, C_{1356}, C_{1456}.
\]
This gives a contradiction, because there are at most $4$ tetra-tangent conics to $D$ in the family by Remark \ref{remquad}. 

Finally assume that $(a,b_1, \dots, b_6)=(-3,2,2,2,2,2,1)$ and consider an exact sequence
\begin{equation}\label{eqa552}
0\to \Oo_S\left(-3L+\sum_{j=1}^5 2E_j+E_6 \right) \to \Omega_S^1(\log D) \to \Ii_{\Gamma, S}\left(2L-\sum_{j=1}^5 E_j\right) \to 0,
\end{equation}
where $\Gamma$ is a zero-dimensional subscheme of $S$ with length $3$, say $\Gamma=\{p,q,r\}$. Similarly as above, by considering the restriction of the sequence (\ref{eqa552}) to each $E_j$, we may assume that $q,r \in E_6$ and $p\not\in E_j$ for any $j$. Applying the push-forward functor $\pi_*$ to the sequence (\ref{eqa552}), we get
\begin{equation}\label{eqa553}
0\to \Oo_{\PP^2}(-3) \to \Omega_{\PP^2}^1(\log \pi(D)) \cong \Tt_{\PP^2}(-2) \to \Ii_{\Gamma', \PP^2}(2) \to 0,
\end{equation}
where $\Gamma'=\{p_1, \dots, p_5, \pi(p), \pi(q)=\pi(r)\}$. Using Lemma \ref{keylem}, we get that the six points $Z'=\{p_1, \dots, p_5, \pi(p)\}$ are not contained in any conic so that we can consider the blow-up of $\PP^2$ at $Z'$ to obtain another smooth cubic surface $\pi': S'\rightarrow \PP^2$ with the exceptional divisors $E_1', \dots, E_5', E_p'$. Now taking the pull-back of the sequence (\ref{eqa553}) together with the logarithmic bundle $\Omega_{S'}^1(\log D')$, where $D'$ is the strict transform of $\pi(D)$ in $S'$, we obtain the following commutative diagram
\begin{equation}\label{dddddia}
\begin{tikzcd}[
  row sep=small,
  ar symbol/.style = {draw=none,"\textstyle#1" description,sloped},
  isomorphic/.style = {ar symbol={\cong}},
  ]
&0\ar[d] &0\ar[d] & \\
0\ar[r] & \Oo_{S'}(-3L') \ar[r] \ar[d] & (\pi')^* \Omega_{\PP^2}^1(\log \pi(D))  \ar[d]\ar[r] &\Ii_{q',S'}(2L'-E') \ar[r]\ar[d,isomorphic] &0\\
0\ar[r] &\Oo_{S'}(-3L'+2E')  \ar[r] \ar[d] & \Omega_{S'}^1(\log D') \ar[r] \ar[d] &\Ii_{q',S'}(2L'-E') \ar[r] & 0\\
& \Oo_{E'}(2E') \ar[r,isomorphic]\ar[d] &\Oo_{E'}(2E') \ar[d] &   \\
&0&0,
\end{tikzcd}
\end{equation}
where $L'$ is a divisor class in $(\pi')^*\Oo_{\PP^2}(1)$, $E'=E_1'+\dots+E_5'+E_p'$, and $q'=(\pi')^{-1}(\pi(q))$. From the previous argument for $(a, b_1, \dots, b_6)=(-3,2,2,2,2,2,2)$, the middle horizontal sequence cannot occur, a contradiction. 
\end{proof}

\subsection{Curves from the linear system $|\Oo_S(3L)|$}\label{subsec-deg9} In this last subsection we want to extend the analogy with the last of the cases studied in Section \ref{sec-projplane}, namely for cubic curves $C\subset \PP^2$. Recall that we proved in Theorem \ref{STP}
that the Torelli property does not hold for pairs $(C,Z)$ of cubic curves $C$ of Sebastiani-Thom type
and a set of points $Z\subset C$, unless $Z$ determines uniquely the curve $C$. Now we are going to prove that the situation is completely opposite when $C\cap Z=\emptyset$. We are going to keep using the following notation from Remark 4.10: recall that $\mathrm{ST}(3)_0$ denotes the set of Sebastiani-Thom type cubics defined as
\[
\mathrm{ST}(3)_0 = \left\{ V(x^3-y^3+az^3) \:|\: a\in \CC^{\times} \right\}.
\]
Consider also the set of three points $W=\{[1:0:0], [0:1:0], [0:0:1]\}$, which do not belong to any curve in $\mathrm{ST}(3)_0$. Let $Z\subset\PP^2$ be the union of $W$ with another three generic points in $\PP^2$ and let $\pi : S =\mathrm{Bl}_{Z}(\PP^2) \rightarrow \PP^2$ be the cubic surface obtained blowing-up $\PP^2$ at $Z$.

\bigskip

\begin{proposition}\label{prop52}
For any two distinct cubic curves $D_1 , D_2\in \mathrm{ST}(3)_0$, we have
\[
\Omega_{S}^1(\log \widetilde{D}_1) \not \cong \Omega_{S}^1(\log \widetilde{D}_2),
\]
where $\widetilde{D}_i$, $i=1,2$, denotes the strict transform of $D_i$ along the blow-up $\pi : S =\mathrm{Bl}_{Z}(\PP^2) \rightarrow \PP^2$.
\end{proposition}

\begin{proof}
Fix a curve $D\in  \mathrm{ST}(3)_0$, and pick an irreducible curve $C \in |\Oo_{S}(L-E_i)|$ for each $i=1,2,3$. Applying Lemma \ref{keylem} to $\Tt_{S}(-\log \widetilde{D})$, we obtain
\[
\left (\Omega_{S}^1(\log \widetilde{D}) \right)_{|C} \cong \Oo_C(1)\oplus \Oo_C,
\]
unless the intersection $C\cap \widetilde{D}$ supports only one point, when the restriction becomes $\Oo_C(2)\oplus \Oo_C(-1)$. In other words, we can recover the nine inflection lines of $D$ from the logarithmic bundle $\Omega_{S}^1(\log \widetilde{D})$. Thus the assertion follows from the fact that the set of nine inflection lines determines uniquely the cubic curve. 
\end{proof}

\section{Final remarks and further questions}
In this section we would like to underline a couple of research lines that, in our opinion, naturally arise from this work and would be interesting to study.\\

(i) One direction of study is to give a complete description of the map 
    \[
    \Psi : |\Oo_X(D)| \dashrightarrow \mathbf{M}_{X}(c_1, c_2)
    \]
    between the linear system of given divisor $D$ in a surface $X$ and the moduli space of stable sheaves (with the appropriate Chern classes), defined by assigning the logarithmic vector bundle. In this work, we have studied the stability and the Torelli property in order to ensure the existence and the generic injectivity of the map. Nevertheless, it would be also interesting to determine and describe its indeterminacy locus and propose an extension of the map itself on the whole linear system.\\
    
(ii) One of the main ideas promoted in this article is that the splitting type of a logarithmic sheaf on a rational curve is special whenever the curve is somewhat tangent to the fixed divisor. On the other hand, we have seen examples in which this special behaviour is not detected; see Example \ref{exmm12}. Nevertheless, we provide evidence that, after blowing up the ambient surface, this special behaviour is always detected. This leads us to proposing the following.

\begin{question} 
Given any reduced Cartier divisor $D$ on a smooth surface $X$, is it possible to always find a finite set of points $Z\subset X\setminus D$ such that, being $\pi : \widetilde{X} \rightarrow X$ the blow-up along $Z$ and $\widetilde{D}$ the strict transform of $D$, $\widetilde{D}$ satisfies the Torelli property?
\end{question}

(iii) Finally, another possible interesting direction is to deal with the higher dimensional cases. The first issue would be to give an extended definition of \textit{generalized logarithmic sheaf}, taking into consideration all the possible variations that depend on the choice of the center of the blow-up. Moreover, it would be also necessary to obtain a generalization of the Key Restriction Lemma \ref{keylem}, which does not only extend it in dimension but also loose the hypothesis on the divisor.

\bibliographystyle{amsplain} 
%\GATHER{MultiplierNotes.bib}
\providecommand{\bysame}{\leavevmode\hbox to3em{\hrulefill}\thinspace}
\providecommand{\MR}{\relax\ifhmode\unskip\space\fi MR }
% \MRhref is called by the amsart/book/proc definition of \MR.
\providecommand{\MRhref}[2]{%
  \href{http://www.ams.org/mathscinet-getitem?mr=#1}{#2}
}
\providecommand{\href}[2]{#2}

\end{document}